\documentclass[onefignum,onetabnum]{siamonline190516}

\usepackage[T1]{fontenc}      
\usepackage[utf8]{inputenc}   
\usepackage{microtype}        
\usepackage{amsmath}
\usepackage{amssymb}
\usepackage{mathtools}
\usepackage{commath}
\usepackage{bbm}
\usepackage{verbatim}
\usepackage{natbib}
\usepackage{cleveref}

\hypersetup{
    colorlinks = false,
}

\newtheorem{assumption}[theorem]{Assumption}
\newsiamremark{remark}{Remark}

\newcommand{\R}{\mathbb{R}} 
\newcommand{\N}{\mathbb{N}} 

\newcommand{\e}{\mathrm{e}} 

\DeclareMathOperator*{\argmin}{arg\,min}  
\DeclareMathOperator*{\argmax}{arg\,max}  

\DeclareMathOperator*{\supp}{supp} 
\DeclareMathOperator*{\inte}{int} 

\newcommand{\defeq}{\coloneqq} 

\newcommand{\textsm}[1]{\textup{\tiny{#1}}} 
\newcommand{\ML}{\textup{\tiny ML}} 
\newcommand{\GP}{\textup{\tiny GP}}
\newcommand{\BC}{\textup{\tiny BC}}

\newcommand{\T}{\mathsf{T}} 
\renewcommand{\b}[1]{#1} 

\DeclarePairedDelimiterX\Set[2]{\lbrace}{\rbrace}%
{ #1 \,:\, #2 } 
\DeclarePairedDelimiterX\inprod[2]{\langle}{\rangle}%
{ #1 , #2 } 
\DeclarePairedDelimiter\ceil{\lceil}{\rceil}
\DeclarePairedDelimiter\floor{\lfloor}{\rfloor}

\begin{document}

\title{Maximum Likelihood Estimation and Uncertainty Quantification for Gaussian Process Approximation of Deterministic Functions\thanks{Submitted to arXiv on 11 May, 2020.
    \funding{TK and FT were supported by the Aalto ELEC Doctoral School. GW was supported by an EPSRC Industrial CASE award [18000171] in partnership with Shell UK Ltd. TK and CJO were supported by the Lloyd's Register Foundation programme on data-centric engineering at the Alan Turing Institute, United Kingdom. SS was supported by the Academy of Finland.
}}}

\headers{GP approximation of deterministic functions}{T. Karvonen, G. Wynne, F. Tronarp, C. J. Oates, and S. Särkkä}

\author{
  Toni Karvonen\thanks{Department of Electrical Engineering and Automation, Aalto University, Finland \& the Alan Turing Institute, UK}
  \and
  George Wynne\thanks{Department of Mathematics, Imperial College London, UK}
  \and
  Filip Tronarp\thanks{Department of Electrical Engineering and Automation, Aalto University, Finland \& University of T\"ubingen, Germany}
  \and
  Chris J. Oates\thanks{School of Mathematics, Statistics and Physics, Newcastle University, UK \& the Alan Turing Institute, UK}
  \and Simo Särkkä\thanks{Department of Electrical Engineering and Automation, Aalto University, Finland}
  }

\maketitle

\begin{abstract} 
Despite the ubiquity of the Gaussian process regression model, few theoretical results are available that account for the fact that parameters of the covariance kernel typically need to be estimated from the dataset.
This article provides one of the first theoretical analyses in the context of Gaussian process regression with a noiseless dataset.
Specifically, we consider the scenario where the scale parameter of a Sobolev kernel (such as a Mat\'{e}rn kernel) is estimated by maximum likelihood.
We show that the maximum likelihood estimation of the scale parameter alone provides significant adaptation against misspecification of the Gaussian process model in the sense that the model can become ``slowly'' overconfident at worst, regardless of the difference between the smoothness of the data-generating function and that expected by the model.
The analysis is based on a combination of techniques from nonparametric regression and scattered data interpolation.
Empirical results are provided in support of the theoretical findings.
\end{abstract}



\section{Introduction} \label{sec: intro}

This article considers the related tasks of approximation and integration of a \emph{deterministic} function $f \colon \Omega \to \R$, defined on $\Omega \subset \R^d$, using Gaussian process (GP) regression based on a noiseless dataset $\mathcal{D} \coloneqq \{(x_n, f(x_n))\}_{n=1}^N$.
In GP regression the true function $f$ is formally considered unknown and is modelled {\it a priori} with a GP $f_\GP \sim \mathrm{GP}(m, K)$, which is characterised by a \emph{mean} function $m : \Omega \rightarrow \mathbb{R}$ and a symmetric positive (semi-)definite covariance function $K \colon \Omega \times \Omega \rightarrow \mathbb{R}$, called a \emph{kernel}. 
The GP is conditioned on the dataset $\mathcal{D}$ and the conditional GP is used to produce credible sets for quantities of interest, such as the function~$f$ itself or its integral.
The popularity of the GP model can be attributed, at least in part, to its elegance, flexibility and computational tractability, and as such GPs underpin much of the modern statistical toolkit for both regression and classification~\citep{RasmussenWilliams2006}.
In the last decade GPs have been adopted in a wide variety of applications, a selection of which includes time series analysis~\citep{Wang2006}, astrophysical data analysis~\citep{Rajpaul2015}, spatial statistics~\citep{Lindgren2011}, bioinformatics~\citep{Gao2008}, robotics~\citep{Yang2013}, functional data analysis~\citep{Shi2008}, computer science~\citep{Manogaran2018}, emulation of computer models~\citep{Sacks1989,Kennedy2001}, and probabilistic numerical computation~\citep{Larkin1972,Hennig2015,Cockayne2019}.

The GP model is typically \emph{misspecified}: the deterministic data-generating function $f$ is not, or does not ``resemble'', a sample path of $f_\GP$.
Accordingly, the critical importance of selecting an appropriate kernel $K$ in GP regression is well-understood~\citep{MacKay1992}.
Different approaches include selecting a single kernel from a continuously parametrised family $\{K^\theta\}_{\theta \in \Theta}$~\citep[Chapter~5]{RasmussenWilliams2006}, selecting a kernel from an arbitrarily rich dictionary of possibilities~\citep{Duvenaud2014,Sun2018}, or even learning a kernel in a nonparametric manner from the data itself~\citep{Bazavan2012,Oliva2016}.
In the parametric case, maximisation of (marginal) likelihood is the most common way to select the kernel parameter $\theta$, for example being the default in well-documented software packages~\citep[e.g.,][]{RasmussenWilliams2006}. 
Despite their ubiquity in the applied context, little is known about the circumstances in which these approaches to kernel parameter selection work well and, by extension, when the credible sets arising from the GP regression model can be trusted. 
The increasing use of GP regression models and their associated credible sets in strategic and safety-critical systems, such as monitoring mine gas emissions \citep{Dong2012}, assessing the health of lithium-ion batteries \citep{Liu2013}, and detecting anomalous or malicious maritime activity \citep{Kowalska2012}, as well as in more general adaptive numerical computation routines~\citep[e.g.,][]{Rathinavel2019}, has led to an urgent need to better understand approaches to kernel parameter selection and model misspecification at a theoretical level.

This article shows that one of the simplest and most commonly used techniques, maximum likelihood estimation of a single \emph{scale parameter} of the kernel, provides a certain amount of protection against model misspecification.
We consider a kernel \sloppy{${K^\sigma(x, y) \defeq \sigma^2 K( x, y)}$} that depends on a scale parameter $\sigma > 0$ and analyse the asymptotic (as $N \to \infty$) behaviour of $\sigma_\ML(f,X_N)$, the \emph{maximum likelihood estimate} of $\sigma$ given noiseless evaluations of~$f$ at a set $X_N \subset \Omega$ consisting of $N$ points, and implications on the coverage of the credible sets derived from the fitted GP model.
For finitely smooth kernels (e.g., Matérn) we show that the maximum likelihood estimate detects the smoothness of the data-generating function: if $K$ induces a Sobolev space of smoothness $\alpha$, $f$ is in a certain sense \emph{exactly} of smoothness $\beta \leq \alpha$, and the points $X_N$ cover $\Omega$ in a sufficiently uniform manner, then $\sigma_\ML(f, X_N)$ is of order $N^{(\alpha-\beta)/d - 1/2}$, up to logarithmic factors.
Because $f$ being akin to a sample of $f_\GP$ roughly speaking corresponds to $\beta = \alpha - d/2$ (see Section~\ref{sec:sample-paths}), the maximum likelihood estimate inflates the conditional variance if $f$ is rougher than the samples and deflates it $f$ is smoother than the samples.
If $f$ is in the Sobolev space of smoothness $\beta \geq \alpha$, then $\sigma_\ML(f, X_N)$ is of order $N^{-1/2}$.
We then use these result to prove that, no matter the degree of over or undersmoothing of $f$ by the kernel, the model can become at most ``slowly'' overconfident in that the GP conditional standard deviation can decay at most with rate $N^{-1/2}$ faster than the true estimation error.
If the scale parameter is held fixed~\eqref{eq:overconfidence-fixed-sigma} demonstrates that the model may become significantly more overconfident than this.

The results are reviewed in more detail in Section~\ref{subsec: main results}.
Section~\ref{sec:general-kernels} considers the case where~$f$ is an element of the reproducing kernel Hilbert space of the kernel $K$ and therefore smoother than expected by the GP. Section~\ref{sec:sobolev-kernels} extends the results for kernels that induce Sobolev spaces by allowing the function to live in a rougher Sobolev space than the one induced by the kernel, in which case the results are dependent on the degree of oversmoothing by the kernel.
Numerical examples are used to validate the theoretical results in Section~\ref{sec:examples}.
In most applications of GP regression there will be several kernel parameters in addition of the scale parameter that must be jointly estimated; our analysis does not extend to that more general setting.
The opportunities and challenges associated with estimation of other kernel parameters are discussed in Section~\ref{sec: discussion}.

\section{Background} \label{sec:background}

In this section we introduce the GP regression model and recall how the kernel scale parameter can be estimated using maximum likelihood. Then we discuss how credible sets can be obtained based on the fitted GP model and what it means to say that the model is asymptotically underconfident or overconfident.
For the latter, we focus on credible sets both for function values and integrals of the function of interest.

\subsection{Gaussian Process Regression} \label{subsec: GP regression}

Let $\Omega$ be an arbitrary subset of $\R^d$ and $f \colon \Omega \to \R$ a deterministic function of interest. 
In GP regression the function $f$ is modelled using a GP $f_\GP$, for which $(f_\GP(x_1), \dots, f_\GP(x_N))$ is Gaussian-distributed for any finite collection $\{x_1,\dots,x_N\} \subset \Omega$ of points.
Let $\mathbb{P}$ denote the law of the GP and let $\mathbb{E}$, $\mathbb{V}$, and $\mathbb{C}$, respectively, denote the expectation, variance, and covariance with respect to $\mathbb{P}$.
The law $\mathbb{P}$ of a GP is characterised by a mean function $m : \Omega \rightarrow \mathbb{R}$, such that $m(x) = \mathbb{E}[f_\GP(x)]$ for all $x \in \Omega$, and a symmetric positive definite covariance function $K : \Omega \times \Omega \rightarrow \mathbb{R}$, called a kernel, such that $K(x,y) = \mathbb{C}[f_\GP(x), f_\GP(y)]$ for all $x,y \in \Omega$. 
Although the kernel can be allowed to be positive semi-definite, in this article we only consider positive definite kernels.
It is common to denote the GP via the shorthand $f_\GP \sim \mathrm{GP}(m, K)$.
Throughout the article and without a loss of generality\footnote{If $m$ is non-zero then the true function $f$ can just be replaced by $f - m$, since $m$ is considered to be known and can thus $f - m$ can also be pointwise evaluated.} we assume that $f_\GP$ is centred (i.e., $m(x) = 0$ for all $x \in \Omega$). 
Further details on GP regression can be found in \cite{Bogachev1998,Stein1999}; and \cite{RasmussenWilliams2006}.

Given a set \sloppy{${\mathcal{D} = \{ (\b{x}_i, f(\b{x}_i))\}_{i=1}^N}$} consisting of \emph{exact} evaluations of $f$ at distinct points $X = \{\b{x}_1, \ldots, \b{x}_N \} \subset \Omega$, the conditional process is again Gaussian: $f_\GP \mid \mathcal{D} \sim \mathrm{GP}( s_{f,X}, P_X )$, with the conditional mean and covariance functions
\begin{equation} \label{eq:GP-posterior}
  s_{f,X}(\b{x}) \coloneqq \b{k}_X(\b{x})^\T \b{K}_X^{-1} \b{f}_X \quad \text{ and } \quad P_X(\b{x}, y) \coloneqq K(\b{x}, y) - \b{k}_X(\b{x})^\T \b{K}_X^{-1} \b{k}_X(y),
\end{equation}
where $(\b{k}_X(\b{x}))_i = K(\b{x}, \b{x}_i)$, $(\b{K}_X)_{i,j} = K(\b{x}_i, \b{x}_j)$, and $(\b{f}_X)_i = f(\b{x}_i)$.
The conditional process quantifies the uncertainty associated with $f$ after the data $\mathcal{D}$ have been observed and can be summarised in terms of a \emph{credible set} for a quantity of interest. 
Let $F$ stand for the cumulative distribution function of the standard normal distribution and denote $\psi_a \coloneqq F^{-1}(1-a/2)$. For any $0 < a < 1$, the Gaussian model implies that
\begin{equation} \label{eq:interval}
  \mathbb{P}\Big[ \abs[1]{ f_{\GP}(\b{x}) - s_{f,X}(\b{x}) } \leq \psi_a P_X(\b{x}, \b{x})^{1/2} \: \Bigl\lvert \: \mathcal{D} \Big] = 1-a \quad \text{ for any } \quad \b{x} \in \Omega.
\end{equation}
Thus (if $P_X(x,x) \neq 0$) the interval bounded by $s_{f,X}(x) \pm \psi_a P_X(x,x)^{1/2}$ is a $(1-a)$ credible set for the unknown quantity $f(x)$ at fixed $x \in \Omega \setminus X$ under the GP model.
However, as is evident from its algebraic expression in~\eqref{eq:GP-posterior}, the conditional covariance $P_X$ does \emph{not} depend on the function evaluations $f_X$, which is clearly undesirable as this implies that the size of the credible set is identical for two wildly different functions evaluated at the same inputs $X$. 
It is well-understood that, for sensible uncertainty quantification to be performed, the kernel should be adapted to the dataset~\citep{MacKay1992}.
When the kernel is parametrised by a collection of parameters $\b{\theta}$ (i.e., $K = K^{\b{\theta}}$), this means that $\b{\theta}$ should be estimated based on the dataset.
Standard approaches to estimation of $\theta$ are reviewed in Section~\ref{subsec: scale param est}.

\subsection{Bayesian Cubature} \label{subsec: BC}

It is convenient to consider and easier to visualise credible sets for scalar quantities derived from $f$, rather than $f$ itself.\footnote{Indeed, unlike the scalar case there is no general consensus on how one should aim to construct a credible set in a function space; see for example \cite{Leibl2019}.}
Moreover, approximation of integrals (i.e., numerical integration) is among the most prevalent applications where noiseless data are provided.
For these reasons we also focus on integrals of $f$ as scalar quantities of interest.
The use of the GP regression model as a means to perform numerical integration is called \emph{Bayesian cubature} (\emph{quadrature} if $d=1$) and is due to \cite{Larkin1972}.
See also \cite{OHagan1991} and \cite{Briol2019} for background.
For a Lebesgue measurable\footnote{Whenever Bayesian cubature is discussed or results for it are provided, it is implicitly assumed in this article that $\Omega$ is Lebesgue measurable.} $\Omega \subset \R^d$ and a positive, bounded, and measurable weight function $w \colon \Omega \to \R$ we consider the integral
\begin{equation} \label{eq:integral}
  I(f) \defeq \int_\Omega f(x) w(x) \dif x
\end{equation}
as a scalar quantity of interest.
Because the integration operator is a linear functional, the random variable $I(f_\GP) \mid \mathcal{D}$ is Gaussian if $\int_\Omega K(x,x) w(x) \dif x < \infty$. Its mean and variance are
\begin{subequations}
\begin{align}
  Q_X(f) & \defeq \mathbb{E}\big[ I(f_\GP) \mid \mathcal{D} \big] = \int_\Omega s_{f,X}(x) w(x) \dif x, \\
  V_X & \defeq \mathbb{V}\big[ I(f_\GP) \mid \mathcal{D} \big] = \int_\Omega \int_\Omega P_X(x,y) w(x) w(y) \dif x \dif y. \label{eq:BQ-var}
\end{align}
\end{subequations}
The Gaussian model for the integral then implies that
\begin{equation} \label{eq:interval-BQ}
  \mathbb{P}\Big[ \abs[1]{ I(f_{\GP}) - Q_X(f) } \leq \psi_a V_X^{1/2} \: \Bigl\lvert \: \mathcal{D} \Big] = 1-a
\end{equation}
and thus (if $V_X \neq 0$) the interval bounded by $Q_X(f) \pm \psi_a V_X^{1/2}$ is a $(1-a)$ credible set, or \emph{credible interval} for $I(f)$ under the GP model.

\subsection{Scale Parameter Estimation} \label{subsec: scale param est}

In this section we consider the case where the kernel $K^\sigma(x,y) = \sigma^2 K(x,y)$ depends on a single fixed scale parameter $\sigma > 0$.
Under the law \sloppy{${f_\GP \sim \mathrm{GP}(0, K^\sigma)}$} the conditional distribution is $f_\GP \mid \mathcal{D} \sim \mathrm{GP}( s_{f,X}, P_X^\sigma)$, where the conditional mean remains unchanged from~\eqref{eq:GP-posterior} and the covariance is
\begin{equation*}
  P_X^\sigma(x, y) \coloneqq \sigma^2 P_X(x, y) = \sigma^2 \big[ K(x, y) - \b{k}_X(\b{x})^\T \b{K}_X^{-1} \b{k}_X(y) \big].
\end{equation*}
The purpose of this article is to analyse the maximum likelihood estimate, $\sigma_\ML(f,X)$, of~$\sigma$ and its effect on the credible sets~\eqref{eq:interval} and~\eqref{eq:interval-BQ}. 
The MLE is defined as the maximiser
\begin{equation} \label{eq:MLE-definition}
  \sigma_\ML(f,X) \coloneqq \argmax_{ \sigma > 0} L(\sigma \mid \mathcal{D}) = \sqrt{ \frac{ \b{f}_X^\T \b{K}_X^{-1} \b{f}_X}{N}}.
\end{equation}
of the log marginal likelihood,
\begin{equation*}
  \log L( \sigma \mid \mathcal{D} ) \defeq -\frac{1}{2} \bigg( \frac{\b{f}_X^\T \b{K}_X^{-1} \b{f}_X}{\sigma^2} + N \log \sigma^2 + \log \det \b{K}_X + N \log (2\pi) \bigg).
\end{equation*}
Equation~\eqref{eq:MLE-definition} is easy to verify by finding the root of the derivative of $L(\sigma \mid \mathcal{D})$.
The estimator $\sigma_\ML(f,X)$ is sometimes called a \emph{maximum marginal likelihood} or \emph{empirical Bayes} estimator.
In applications where additional parameters are present in the kernel, these could be simultaneously estimated based on the dataset.
However, our focus on the scale parameter is due to the closed-form expression in~\eqref{eq:MLE-definition}; such expressions are not be available in general.

\subsection{Credible Sets and Maximum Likelihood} \label{sec:parameter-estimation}

Adopting the maximum likelihood approach to parameter selection means that $\sigma$ is replaced by $\sigma_\ML(f,X)$ in~\eqref{eq:interval} and~\eqref{eq:interval-BQ} to produce
\begin{align*}
  \mathbb{P}\Big[ \abs[1]{ f_{\GP}(\b{x}) - s_{f,X}(\b{x}) } \leq \psi_a \sigma_\ML(f,X) P_X(\b{x}, \b{x})^{1/2} \: \Bigl\lvert \: \mathcal{D} \Big] &= 1-a, \\
  \mathbb{P}\Big[ \abs[1]{ I(f_{\GP}) - Q_X(f) } \leq \psi_a \sigma_\ML(f,X) V_X^{1/2} \: \Bigl\lvert \: \mathcal{D} \Big] &= 1-a.
\end{align*}
We use the compact notation
\begin{equation} \label{eq:credible-widths}
  R_\GP(\b{x}, f,X) \coloneqq \sigma_\ML(f,X) P_X(\b{x}, \b{x})^{1/2} \quad \text{ and } \quad R_\BC(f,X) \coloneqq \sigma_\ML(f,X) V_X^{1/2}
\end{equation}
for the unscaled widths of the credible sets and denote the credible sets as
\begin{subequations} \label{eq:credible-sets}
\begin{align}
  \mathcal{C}_\GP^a(x,f,X) &\coloneqq \Set[\Bigg]{ y \in \R }{ \frac{\abs[0]{y - s_{f,X}(x)}}{R_\GP(x,f,X)} \leq \psi_a }, \label{eq:credible-set-GP}, \\
  \mathcal{C}_\BC^a(f,X) &\coloneqq \Set[\Bigg]{ \mu \in \R}{ \frac{\abs[0]{\mu - Q_X(f)}}{R_\BC(f,X)} \leq \psi_a }. \label{eq:credible-set-BQ}
\end{align}
\end{subequations}
These credible sets underpin inferences and decisions based on the fitted GP regression model, with applications in diverse fields, including strategic and safety-critical systems, several of which were mentioned in Section~\ref{sec: intro}.
It is therefore important to understand when these sets can and cannot be trusted to accurately reflect the function $f$ or its integral.

It is immediately clear from \eqref{eq:MLE-definition} that credible sets are invariant to scaling of $f$, in the sense that the transformation $f \mapsto \lambda f$ for some constant $\lambda$ leads to $\sigma_\ML(f,X) \mapsto \abs[0]{\lambda} \sigma_\ML(f,X)$.
However, it is far from clear how these credible sets behave as a function of the point set $X$.
In particular, we consider the limit of a large number of points next.

\subsection{Asymptotics of Credible Sets} \label{subsec: coverage of credible intervals}

Consider a sequence $(X_N)_{N = 1}^\infty \subset \Omega$ of point sets such that $X_N$ contains $N$ distinct points.
The function $f$ is fixed and our focus is on the behaviour of credible sets when $N \rightarrow \infty$, a setting called \emph{fixed domain asymptotics} by \cite{Stein1999}. Specifically, we are interested in whether or not $f(x)$ or $I(f)$ can be expected to fall within the relevant credible set, $\mathcal{C}^a_\GP(x,f,X_N)$ or $\mathcal{C}^a_\BC(f,X_N)$, for large $N$.
To avoid confusion, it is important to note that our focus is distinct from the assessment of \emph{frequentist coverage} that is more commonplace in the statistical literature. There, it is most common for $N$ and $f$ to be fixed and for observations of $f$ to be contaminated with noise; one can then ask for credible sets to have correct coverage with respect to realisations of the noise generating process.
Equally, our analysis is distinct from an assessment of frequentist coverage in which $f$ is considered to be drawn at random from $\mathbb{P}$ and observed (without noise) at $N$ locations.
To emphasise, in this article the dataset~$\mathcal{D}$ and function $f$ are deterministic and the only source of uncertainty is the \emph{epistemic uncertainty} from the GP regression model.

We say that a GP model with a covariance kernel $K$ is \emph{asymptotically overconfident} for approximation at $x \in \Omega$ (respectively, integration) of a function $f \colon \Omega \to \R$ if
\begin{equation} \label{eq:overconfident}
  \liminf_{N \to \infty} \frac{ \abs[0]{f(x) - s_{f,X_N}(x)}}{ R_\textsm{GP}(x,f,X_N)} = \infty \quad \quad \Bigg( \liminf_{N \to \infty} \frac{\abs[0]{I(f) - Q_{X_N}(f)}}{R_\textsm{BC}(f,X_N)} = \infty \Bigg)
\end{equation}
and \emph{asymptotically underconfident} if
\begin{equation} \label{eq:underconfident}
  \lim_{N \to \infty} \frac{ \abs[0]{f(x) - s_{f,X_N}(x)}}{ R_\textsm{GP}(x,f,X_N)} = 0 \quad \quad \Bigg( \lim_{N \to \infty} \frac{\abs[0]{I(f) - Q_{X_N}(f)}}{R_\textsm{BC}(f,X_N)} = 0 \Bigg).
\end{equation}
Conforming to conventional statistical terminology we call the ratios in~\eqref{eq:overconfident} and~\eqref{eq:underconfident} \emph{standard scores}.
Note that $R_\textsm{GP}(x,f,X_N) = 0$ only if $f(x) = s_{f,X_N}(x)$; in this case we set $0/0 = 1$.
Asymptotic overconfidence means that the width of the credible set decays faster than the true approximation or integration error:
for any fixed $a \in (0,1)$ we have $f(x) \notin \mathcal{C}_\GP^a(x,f,X_N)$ or $I(f) \notin \mathcal{C}_\BC^a(f,X_N)$ for all sufficiently large $N$.
Conversely, asymptotic underconfidence implies that for any $a \in (0,1)$ we have $f(x) \in \mathcal{C}_\GP^a(x,f,X_N)$ or $I(f) \in \mathcal{C}_\BC^a(f,X_N)$ for all $N$ large enough.

Overconfidence can have disastrous effect in particular in safety-critical applications while underconfidence results in inefficiency as more data than is necessary is needed to attain the same level of assurance.
The ideal state of affairs is thus for the model to be neither asymptotically overconfident nor underconfident, a situation which we call \emph{asymptotic honest} as this implies that the size of the credible sets decay at rate that is commensurate with the true approximation error. See \cite{Szabo2015} for a similar concept.
In practice asymptotic honesty is a weak requirement and does not guarantee credible sets can be trusted at finite values of $N$.
Our tools are not powerful enough to identify or prove the existence of meaningful collections of functions for which the model is asymptotically honest, and our results concern only asymptotic overconfidence and underconfidence.

\subsection{Prior Work on Maximum Likelihood Estimation} \label{subsec: prior work}

The only prior work in an identical setting, to the best of our knowledge, is by \cite{XuStein2017} and \cite{Karvonen-MLSP2019}.
\cite{XuStein2017} considered the Gaussian kernel $K(x,y) = \exp(-(x-y)^2/(2\ell^2))$ with $\ell > 0$ fixed and monomials $f(x) = x^p$ on $[0,1]$, evaluated at successive sets of $N$ equispaced points, $X_N = \{1/N, 2/N, \ldots, 1\}$. 
They conjectured an asymptotic equivalence
\begin{equation*}
  \sigma_\ML(f, X_N) \sim \frac{\ell^{2p}}{\sqrt{2\pi} (p+1/2)} N^{p - 1/2} \quad \text{ for any } \quad p \geq 0
\end{equation*}
and proved this for $p=0$ and partially for $p=1$ using an explicit Cholesky decomposition of the kernel matrix.
\cite{Karvonen-MLSP2019} worked with the Ornstein--Uhlenbeck kernel $K(x,y) = \exp(-\lambda\abs[0]{x-y}) - \exp(\!-\lambda(x+y))$, with $\lambda > 0$ fixed, and equispaced evaluation points on $[0,1]$. 
They proved that $ \lim_{N \to \infty}\sigma_\ML(f,X_N)$ is proportional to the quadratic variation $V^2(f)$ of $f$. Consequently, the MLE converges to zero if the Hölder exponent of $f$ exceeds $1/2$ (e.g., the function is differentiable) and to a positive constant if $V^2(f) \in (0, \infty)$. As almost all sample paths of the Ornstein--Uhlenbeck process have a finite non-zero quadratic variation, this is in agreement with the intuition that the maximum likelihood estimate should behave reasonably if the function is plausible as a sample from the GP. 

In addition, frequentist coverage of Bayesian credible sets when various hyperparameters of a GP are selected with maximum likelihood has been extensively studied by \cite{Szabo2013,Szabo2015} and \cite{HadjiSzabo2019}.
In these articles the model of interest differs from ours, being the Gaussian \emph{white noise model} for an unknown function $f(x) = \sum_{i=1}^\infty \vartheta_i \varphi_i(x)$ expressed in a basis $\{\varphi_i\}_{i=1}^\infty$. A sequence $Y = (Y_i)_{i=1}^\infty$ of noisy observations are made directly on the square-summable parameter $\vartheta = (\vartheta_i)_{i=1}^\infty$ via
\begin{equation*}
  Y_i = \vartheta_{i} + \frac{1}{\sqrt{\eta}} Z_i, \quad \text{ where } \quad \quad Z_i \sim \mathrm{N}(0,1) \text{ are i.i.d.}
\end{equation*}
In \cite{Szabo2013,Szabo2015} the parameter $\vartheta$ was assigned a Gaussian prior distribution that is analogous to GPs with Sobolev kernels that we analyse.
Behaviour as $\eta \to \infty$ (i.e., the noise level decreases) of the MLE of the scaling parameter of this prior and the coverage properties of the resulting credible sets were analysed in \cite{Szabo2013} for the true parameter satisfying $\vartheta_i^2 \leq C_2^2 i^{-1-2\beta}$ or $C_1^2 i^{-1-2\beta} \leq \vartheta_i^2 \leq C_2^2 i^{-1-2\beta}$ for some $C_1, C_2 > 0$ and a smoothness parameter $\beta > 0$. 
These sets are analogous to our $S_-^\beta(\R^d)$ and $S^\beta(\R^d)$ defined in Section~\ref{sec:sobolev-spaces}.
The white noise model is widely studied as a theoretically tractable analogue of regression with noisy data. As such the results are not directly applicable in our context where the function $f$ is exactly evaluated.

For other work related to GP misspecification and kernel parameter estimation in a variety of settings, see \citet{Stein1993,Bachoc2013,Bachoc2017,Bachoc2018,Bachoc2019}; and \citet{Teckentrup2019}.

\subsection{Our Contributions} \label{subsec: main results}

Let $(X_N)_{N=1}^\infty \subset \Omega$ be a sequence of point sets, each containing $N$ distinct points, and let the function $f \colon \Omega \to \R$ be fixed. Our results concern (i) the behaviour, as $N \to \infty$, of the maximum likelihood estimate, $\sigma_\ML(f,X_N)$ in~\eqref{eq:MLE-definition}, of the GP scale parameter based on exact evaluation of $f$ on $X_N$ and (ii) whether or not this induces asymptotic overconfidence or underconfidence in the GP model, as defined in~\eqref{eq:overconfident} and~\eqref{eq:underconfident}.

\paragraph{Reproducing kernel Hilbert spaces} 

In Section~\ref{sec:general-kernels} we do not place any restrictions on the covariance kernel $K$. 
We first prove the surprising result that if $f$ is an element of $\mathcal{H}_K(\Omega)$, the reproducing kernel Hilbert space of $K$, then
\begin{equation} \label{eq:mle-rkhs-contributions}
  \sigma_\ML(f, X_N) \asymp N^{-1/2}
\end{equation}
regardless of the point sets $X_N$ used, provided the $X_N$ share a common element $x^*$ such that $f(x^*) \neq 0$ (Proposition~\ref{prop:MLE-general}).
Theorem~\ref{thm:coverage-RKHS}, an implication of this, states that for such functions and point sets the model cannot become overconfident ``too fast'', meaning that
\begin{equation} \label{eq:overconfidence-contributions}
  \sup_{x \in \Omega} \, \frac{ \abs[0]{ f(x) - s_{f,X_N}(x)}}{ R_\textsm{GP}(x,f,X_N)} = O(N^{1/2}) \quad \text{ and } \quad \frac{\abs[0]{I(f) - Q_{X_N}(f)}}{ R_\textsm{BC}(f,X_N)} = O(N^{1/2}).
\end{equation}
Note that this does \emph{not} imply that the model is asymptotically overconfident.
Indeed, in Theorem~\ref{thm:underconfidence-RKHS} we show that \emph{underconfidence} occurs if $f$ belongs to a certain subspace of $\mathcal{H}_K(\Omega)$.

\paragraph{Sobolev spaces} Section~\ref{sec:sobolev-kernels} focusses on \emph{Sobolev kernels},  which induce Sobolev spaces and include the popular Matérn kernels.
The restrictive assumption $f \in \mathcal{H}_K(\Omega)$ is relaxed and it is proven in Proposition~\ref{thm:MLE-upper} that if $K$ induces a Sobolev space of smoothness $\alpha$ and $f$ is in the Sobolev space of smoothness $\beta < \alpha$, then
\begin{equation} \label{eq:mle-sobolev-contributions}
  \sigma_\ML(f, X_N) = O\big( N^{(\alpha-\beta)/d-1/2} \big)
\end{equation}
assuming that $X_N$ cover the domain $\Omega$ in a uniform fashion;~\eqref{eq:mle-rkhs-contributions} is applicable  if $\beta \geq \alpha$.
Moreover, a similar lower bound is available when a lower bound on the smoothness of $f$ is known (Proposition~\ref{thm:MLE-lower}).
If it is known that $f$ is of \emph{exact smoothness} $\beta \leq \alpha$, in that it belongs to the set $S^\beta(\Omega)$ in~\eqref{eq:S-definition}, then the rate~\eqref{eq:mle-sobolev-contributions} is sharp up to logarithmic factors (Theorem~\ref{cor:MLE-upper-lower}).
In particular, if $f$ is of exact smoothness $\beta = \alpha - d/2$, which roughly speaking corresponds to $f$ having the same regularity as samples from the GP (see~Section~\ref{sec:sample-paths}), then $\sigma_\ML(f,X_N)$ is constant up to logarithmic factors.
If the exact smoothness of $f$ is known, bounds similar to~\eqref{eq:overconfidence-contributions} on the standard scores then hold by Theorem~\ref{thm:coverage-sobolev}.
These results thus show that maximum likelihood estimation of the scale parameter is a useful tool in adapting the GP model to misspecified smoothness of the data-generating function.
Finally, according to Theorem~\ref{thm:underconfidence-sobolev},~$f$ being much smoother than the kernel implies underconfidence of the GP model.

Empirical results in Section~\ref{sec:examples} verify the MLE asymptotics~\eqref{eq:mle-rkhs-contributions} and~\eqref{eq:mle-sobolev-contributions} but suggest that the standard score bounds~\eqref{eq:overconfidence-contributions} and their extensions in the Sobolev setting are not tight.
Although sufficient conditions for asymptotic honesty of a GP model are not provided here, the collection of results that we establish represents a substantial expansion of what is currently known in the context of maximum likelihood estimation with a noiseless dataset.

\subsection{Notation} \label{subsec: notation}

For $x \in \mathbb{R}^d$ we let $\|x\| \defeq (x_1^2+ \dots + x_d^2)^{1/2}$ be the Euclidean norm. 
The space $L^p(\Omega)$ stands for the space of $p$-integrable functions on a Lebesgue measurable set $\Omega \subset \R^d$.
For non-negative sequences $(a_n)_{n=1}^\infty$ and $(b_n)_{n=1}^\infty$ we denote $a_n \lesssim b_n$ ($a_n \gtrsim b_n$) if there is $C > 0$ such that $a_n \leq C b_n$ ($a_n \geq C b_n$) for every sufficiently large~$n$.
When $a_n \lesssim b_n \lesssim a_n$, we write $a_n \asymp b_n$.
Analogous notation is used for non-negative functions. For example, $h(x) \lesssim g(x)$ means that there is $C > 0$ such that $h(x) \leq C g(x)$ for all sufficiently large $\norm[0]{x}$.

The \emph{restriction} of a function $g \colon A \to \R$ on a subset $B \subset A$ is the function $g|_B \colon B \to \R$ such that $g|_B(b) = f(b)$ for every $b \in B$.
In particular, the statement that $h|_X = g|_X$ for a set $X \subset \R^d$ means that the function $g$ \emph{interpolates} $h$ on $X$.
Conversely, if $g \colon A \to \R$ and $h \colon B \to \R$ are such that $g|_B = h$, then $g$ is said to be an \emph{extension} of $h$ (onto $A$).

In what follows the set $X = \{x_1, \ldots,x_N\}$ always denotes a collection of $N \in \mathbb{N}$ distinct points contained in the domain $\Omega \subset \mathbb{R}^d$ of the function $f$ of interest. If it is necessary to emphasise the number of points in the set, we write $X_N$ for a set of $N$ points.

\section{Approximation of Functions in the RKHS} \label{sec:general-kernels}

In this section we introduce reproducing kernel Hilbert spaces (RKHSs) and study the maximum likelihood estimate and implications for the standard scores when $f$ is regular enough to be contained in the RKHS of the covariance kernel.
Results for less regular functions are deferred until Section~\ref{sec:sobolev-kernels}.

\subsection{Positive-Definite Kernels and RKHSs} \label{subsec: RKHS introduce}

The monograph of \citet{BerlinetThomasAgnan2004} is a standard introduction to the theory of RKHS.
Let $\Omega$ be an arbitrary subset of $\R^d$.
We say that a function $K \colon \Omega \times \Omega \to \R$ is a kernel (on $\Omega$) if it is \emph{positive-definite}.
Positive-definiteness entails that, for any $N \in \mathbb{N}$, the $N \times N$ \emph{kernel matrix} $(\b{K}_X)_{i,j} = K(\b{x}_i, \b{x}_j)$ is positive-definite for any set $X = \{ \b{x}_1, \ldots, \b{x}_N \} \subset \R^d$ of $N$ distinct points.
Every kernel induces a unique \emph{reproducing kernel Hilbert space} $\mathcal{H}_K(\Omega)$ equipped with the inner product $\inprod{\cdot}{\cdot}_{\mathcal{H}_K(\Omega)}$ and the induced norm $\norm[0]{\cdot}_{\mathcal{H}_K(\Omega)}$. This space consists of certain sufficiently regular functions $g \colon \Omega \to \R$ and is characterised by
\begin{itemize}
\item[(i)] $K(\cdot, \b{x}) \in \mathcal{H}_K(\Omega)$ for every $\b{x} \in \Omega$ and
\item[(ii)] $\inprod{g}{K(\cdot,x)}_{\mathcal{H}_K(\Omega)} = g(x)$ for every $g \in \mathcal{H}_K(\Omega)$ and $\b{x} \in \Omega$ (the \emph{reproducing property}).
\end{itemize}
Note the RKHS~$\mathcal{H}_K(\Omega)$ and its norm are always those of the ``unscaled'' kernel $K$. That is, they do not depend on the scale parameter $\sigma$.

Throughout the article we assume that $K$ is a kernel.
In this section the kernel $K$ is arbitrary, meaning that it is not necessarily straightforward to verify if a given function is contained in its RKHS.
However, in Section~\ref{sec:sobolev-kernels} the kernel is selected such that the RKHS is a Sobolev space so that the differentiability of a function determines if it is a member of the RKHS.
We occasionally define the kernel on the whole of $\R^d$ and then consider the restriction of $\mathcal{H}_K(\R^d)$ to~$\Omega \subset \R^d$.
The restriction consists of functions $g \colon \Omega \to \R$ that admit an extension $g_0 \in \mathcal{H}_K(\R^d)$ and its norm is
\begin{equation*}
  \norm[0]{g}_{\mathcal{H}_K(\Omega)} \defeq \inf\Set[\big]{ \norm[0]{g_0}_{\mathcal{H}_K(\R^d)} }{ g_0 \in \mathcal{H}_K(\R^d) \text{ such that } g_0|_\Omega = g}.
\end{equation*}

\subsection{Kernel Interpolation and Error Estimates}

It is necessary to recognise the equivalence of GP regression and \emph{kernel} or \emph{radial basis function interpolation}~\citep{Wendland2005,FasshauerMcCourt2015}: the GP conditional mean~\eqref{eq:GP-posterior} is the kernel interpolant to $f$ at $X$, which is to say that it is the unique function $g$ in $\mathrm{span}\{K(\cdot, \b{x}_i)\}_{i=1}^N$ such that $g|_X = f|_X$. Equivalently, $s_{f,X}$ is the interpolant to $f$ of minimal norm among the functions in the RKHS of the kernel:
\begin{equation} \label{eq:s-minimal}
  s_{f,X} = \argmin \Set[\big]{ \norm[0]{g}_{\mathcal{H}_K(\Omega)} }{ g \in \mathcal{H}_K(\Omega) \text{ such that } g|_X = f|_X}.
\end{equation}
This property implies in particular that $\norm[0]{s_{f,X}}_{\mathcal{H}_K(\Omega)} \leq \norm[0]{f}_{\mathcal{H}_K(\Omega)}$ if $f \in \mathcal{H}_K(\Omega)$.
If $f \notin \mathcal{H}_K(\Omega)$, the conditional mean is still an element of the RKHS but its norm diverges to infinity as $X$ becomes denser.
Further discussion on the relationship between GP regression and kernel-based minimum-norm interpolation can be found in \citet{Scheuerer2013,Kanagawa2018,Karvonen2019}; and \citet[Chapter~17]{FasshauerMcCourt2015}. 
\citet[Chapter~3]{Oettershagen2017} contains a compact collection of basic results on approximation in RKHS. 

The RKHS framework is useful in deriving generic estimates for GP approximation or integration error. 
The conditional variances~\eqref{eq:GP-posterior} and~\eqref{eq:BQ-var} are equal to squared \emph{worst-case errors} in function and integral approximation in the RKHS of the covariance kernel:
\begin{equation} \label{eq:WCE}
  P_X(\b{x}, \b{x})^{1/2} = \sup_{ \norm[0]{g}_{\mathcal{H}_K(\Omega)} \leq 1} \, \abs[0]{ g(\b{x}) - s_{g,X}(\b{x}) } \quad \text{ and } \quad V_X^{1/2} = \sup_{ \norm[0]{g}_{\mathcal{H}_K(\Omega)} \leq 1} \abs[0]{ I(g) - Q_X(g) }.
\end{equation}
Furthermore, the reproducing property of the kernel can be used in bounding the approximation or integration error for a specific function $f \in \mathcal{H}(K)$ using the standard deviations:
\begin{equation} \label{eq:WCE-bounds}
  \abs[0]{ f(\b{x}) - s_{f,X}(\b{x}) } \leq \norm[0]{f}_{\mathcal{H}_K(\Omega)} P_X(\b{x}, \b{x})^{1/2} \quad \text{ and } \quad \abs[0]{ I(f) - Q_X(f) } \leq \norm[0]{f}_{\mathcal{H}_K(\Omega)} V_X^{1/2}.
\end{equation}

\subsection{Maximum Likelihood Estimation in the RKHS} \label{subsec: MLEs in RKHS}

In this section we study the maximum likelihood estimator $\sigma_\textsm{ML}(f,X_N)$ and asymptotic underconfidence and overconfidence of the GP model when the function $f$ is sufficiently regular to belong to $\mathcal{H}_K(\Omega)$.\footnote{Note that, as discussed in detail in Section~\ref{sec:sample-paths}, samples from the GP do not lie in this RKHS with probability~1 if the RKHS is infinite-dimensional.}

All results in this article are based on the following expression for the MLE that, simple as it is, appears to have been seldom exploited in the GP literature:
\begin{equation} \label{eq:MLE}
  \sigma_\ML(f,X_N) = \frac{1}{\sqrt{N}} \norm[0]{s_{f,{X_N}}}_{\mathcal{H}_K(\Omega)}.
\end{equation}
Note that this equation does not require that $f \in \mathcal{H}_K(\Omega)$.
This connection between the maximum likelihood estimate of the scale parameter and the RKHS norm of the conditional mean is made explicit in \citet[Remark~9.2]{FasshauerMcCourt2015}, and the straightforward proof, based on the reproducing property and equations~\eqref{eq:GP-posterior} and~\eqref{eq:MLE-definition}, can also be found in, for example, \citet[Section~5.1]{Fasshauer2011}.
\citet[Section~3.3]{Bull2011} uses~\eqref{eq:MLE} in the context of Bayesian optimisation.
Equation~\eqref{eq:MLE} leads immediately to our first result for $f \in \mathcal{H}_K(\Omega)$.

\begin{proposition}[MLE in the RKHS] \label{prop:MLE-general} If $f \in \mathcal{H}_K(\Omega)$, then \sloppy{${\sigma_\textsm{ML}(f,X_N) \leq N^{-1/2} \norm[0]{f}_{\mathcal{H}_K(\Omega)}}$}.
Furthermore, if there exists a point $x^* \in \Omega$ such that $f(x^*) \neq 0$ and $x^* \in X_N$ for all sufficiently large $N$, then $\sigma_\textsm{ML}(f,X_N) \asymp N^{-1/2}$. 
\end{proposition}
\begin{proof} If $f \in \mathcal{H}_K(\Omega)$, it follows from~\eqref{eq:MLE} and the minimum-norm characterisation~\eqref{eq:s-minimal} of the conditional mean that
  \begin{equation*}
      \sigma_\ML(f,X_N) = \frac{\norm[0]{s_{f,X_N}}_{\mathcal{H}_K(\Omega)}}{N^{1/2}} \leq \frac{ \norm[0]{f}_{\mathcal{H}_K(\Omega)} }{ N^{1/2}}.
  \end{equation*}
The minimum-norm characterisation also implies that \sloppy{${0 < \norm[0]{s_{f,\{x^*\}}}_{\mathcal{H}_K(\Omega)} \leq \norm[0]{s_{f,X_N}}_{\mathcal{H}_K(\Omega)}}$} if \sloppy{${x^* \in X_N}$} and $f(x^*) \neq 0$, which proves the second claim and completes the proof.
\end{proof}

The reasonableness or otherwise of this behaviour for the maximum likelihood estimate is best assessed in the context of its implied conditional GP, and in particular the behaviour of its credible sets.

\begin{theorem}[Slow overconfidence at worst in the RKHS] \label{thm:coverage-RKHS} If $f \in \mathcal{H}_K(\Omega)$ and there is $x^* \in \Omega$ such that $f(x^*) \neq 0$ and $x^* \in X_N$ for all sufficiently large $N$, then
  \begin{equation} \label{eq:RKHS-N12}
    \sup_{x \in \Omega} \, \frac{ \abs[0]{ f(x) - s_{f,X_N}(x)}}{ R_\textsm{GP}(x,f,X_N)} \lesssim N^{1/2} \quad \text{ and } \quad \frac{\abs[0]{I(f) - Q_{X_N}(f)}}{ R_\textsm{BC}(f,X_N)} \lesssim N^{1/2}.
  \end{equation}
\end{theorem}
\begin{proof} From~\eqref{eq:WCE-bounds} we know that $\abs[0]{f(x) - s_{f,X_N}(x)} \leq \norm[0]{f}_{\mathcal{H}_K(\Omega)} P_{X_N}(x,x)^{1/2}$ for every $x \in \Omega$ if $f \in \mathcal{H}_K(\Omega)$. By this estimate and Proposition~\ref{prop:MLE-general},
  \begin{equation*}
    \frac{ \abs[0]{f(x) - s_{f,X_N}(x)}}{ R_\textsm{GP}(x,f,X_N)} = \frac{ \abs[0]{f(x) - s_{f,X_N}(x)}}{ \sigma_\textsm{ML}(f,X_N) P_{X_N}(x,x)^{1/2}} \leq \frac{ \norm[0]{f}_{\mathcal{H}_K(\Omega)} }{\sigma_\textsm{ML}(f,X_N)} \asymp N^{1/2}.
  \end{equation*}
  The argument for integration is identical.
\end{proof}

The interpretation of Theorem~\ref{thm:coverage-RKHS} is that a GP model can become at worst slowly overconfident, in the sense that the credible sets are asymptotically $O(N^{1/2})$ times narrower than they would be if the model was asymptotically honest. 
After the present work was completed, a closely related result appeared as Proposition~3.1 in \citet{Wang2020}.

\begin{remark} \cite{Szabo2015} included a blow-up factor $L_N > 0$ in the studied credible sets, which in our setting is equivalent to using the scale parameter \sloppy{${\sigma = L_N \sigma_\ML(f,X_N)}$}. 
If~$L_N$ is set to grow sufficiently fast, our results guarantee that the model is not asymptotically overconfident.
For example, if $L_N \gtrsim N^{1/2}$ a modification of Theorem~\ref{thm:coverage-RKHS} would state that the standard scores are $O(1)$.
It is not clear to us that such artificial inflation of $\sigma$ can be statistically justified.
\end{remark}

Theorem~\ref{thm:coverage-RKHS} establishes only upper bounds on standard scores and it does not follow that there is a function for which the model is asymptotically overconfident---let alone that this is the case for all functions in the RKHS.
In fact, the upper bounds~\eqref{eq:RKHS-N12} can be improved to $\norm[0]{f-s_{f,X_N}}_{\mathcal{H}_K(\Omega)} N^{1/2}$ by the use of improved versions~\citep[e.g.,][p.\ 192]{Wendland2005} of the generic error estimates~\eqref{eq:WCE-bounds}:
\begin{subequations} \label{eq:improved-rates-generic}
  \begin{align}
  \abs[0]{ f(x) - s_{f,X}(x) } &\leq \norm[0]{f-s_{f,X}}_{\mathcal{H}_K(\Omega)} P_X(x,x)^{1/2}, \\
  \abs[0]{ I(f) - Q_X(f) } &\leq \norm[0]{f-s_{f,X}}_{\mathcal{H}_K(\Omega)} V_X^{1/2}.
  \end{align}
\end{subequations}
If the RKHS error $\norm[0]{f-s_{f,X_N}}_{\mathcal{H}_K(\Omega)}$ decays sufficiently fast it can be established that the model is not asymptotically overconfident.
Although it is known that $\norm[0]{f-s_{f,X_N}}_{\mathcal{H}_K(\Omega)} \to 0$ as $N \to \infty$ if the kernel is continuous and the point-set sequence $(X_N)_{N=1}^\infty$ is space-filling (in the sense that the fill-distance, to defined in Section~\ref{sec:Sobolev-rates}, decays to zero), this convergence in the RKHS norm can be arbitrarily slow~\citep[Theorem~8.37 and~Exercise~8.64]{Iske2018}.
It is therefore interesting to ask whether there is a well-characterised subset of the RKHS for which the GP model is not asymptotically overconfident.
Such a subset is identified next.

\subsection{Asymptotic Underconfidence for a Subset of the RKHS}

In this section we characterise a subset of the RKHS, related to an $L^2(\Omega)$ integral operator, where the true approximation error can be shown to decay faster than the width of the credible set.
If $\Omega \subset \R^d$ is compact and the kernel $K$ continuous, it follows that the integral operator $T \colon L^2(\Omega) \to L^2(\Omega)$ defined as
\begin{equation} \label{eq:integral-operator}
  Tg(x) \coloneqq \int_\Omega g(y) K(x,y) \dif y \quad \text{ for } \quad g \in L^2(\Omega)
\end{equation}
is self-adjoint and compact. 
By the spectral theorem there exists a sequence of positive and decreasing eigenvalues $(\lambda_n)_{n=1}^\infty$ and corresponding eigenfunctions $(\varphi_n)_{n=1}^\infty \subset L^2(\Omega)$ such that $T\varphi_n = \lambda_n \varphi_n$. 
Since $K$ is assumed continuous, Mercer's theorem implies that $(\varphi_n)_{n=1}^\infty$ form an orthonormal basis of $L^2(\Omega)$ and $(\lambda_n^{1/2} \varphi_n)_{n=1}^\infty$ form an orthonormal basis of $\mathcal{H}_K(\Omega)$ when each~$\varphi_n$ is uniquely identified with a continuous element of the RKHS.
Therefore the kernel has the uniformly convergent expansion $K(x,y) = \sum_{n=1}^\infty \lambda_n \varphi_n(x) \varphi_n(y)$ on $\Omega \times \Omega$ and the RKHS is
\begin{equation*}
  \mathcal{H}_K(\Omega) = \Set[\Bigg]{ g \in L^2(\Omega)}{ \sum_{n=1}^\infty \frac{\inprod{g}{\varphi_n}_{L^2(\Omega)}^2}{\lambda_n} < \infty}.
\end{equation*}
It can be then shown that the range of $T$ is 
\begin{equation} \label{eq:T-image}
  T(L^2(\Omega)) = \Set[\Bigg]{ g \in L^2(\Omega)}{ \sum_{n=1}^\infty \frac{\inprod{g}{\varphi_n}_{L^2(\Omega)}^2}{\lambda_n^2} < \infty } \subset \mathcal{H}_K(\Omega).
\end{equation}
It is easy to prove that the error estimates~\eqref{eq:WCE-bounds} can be improved if $f$ is in $T(L^2(\Omega))$~\citep[Section~11.5]{Wendland2005}. Namely, if there is $v \in L^2(\Omega)$ such that $f = Tv$, then
\begin{equation*}
  \norm[0]{f - s_{f,X}}_{\mathcal{H}_K(\Omega)} \leq \bigg( \int_\Omega P_X(x,x) \dif x \bigg)^{1/2} \norm[0]{v}_{L^2(\Omega)} \eqqcolon \norm[0]{P_X^{1/2}}_{L^2(\Omega)} \norm[0]{v}_{L^2(\Omega)}
\end{equation*}
and therefore by~\eqref{eq:improved-rates-generic} the error estimates become
\begin{subequations} \label{eq:improved-rates}
  \begin{align}
    \abs[0]{f(x) - s_{f,X}(x)} &\leq P_{X}(x,x)^{1/2} \norm[0]{P_X^{1/2}}_{L^2(\Omega)} \norm[0]{v}_{L^2(\Omega)}, \\
      \abs[0]{I(f) - Q_X(f)} &\leq V_X^{1/2} \norm[0]{P_X^{1/2}}_{L^2(\Omega)} \norm[0]{v}_{L^2(\Omega)}.
    \end{align}
\end{subequations}
The standard convergence rates are thus effectively squared, this being occasionally referred to as \emph{superconvergence}~\citep{Schaback2018}. See \cite{Schaback1999,Schaback2000}; \citet[Section~9.4.3]{FasshauerMcCourt2015}; and \citet[Section~5]{Bach2017} for additional results and discussion and \citet[Section~6.2]{Kanagawa2019} for numerical examples.
Also note the connection of the space~\eqref{eq:T-image} to powers of RKHSs~\citep{SteinwartScovel2012} and Hilbert scales~\citep[Appendix~A.1.3]{DashtiStuart2017}.
Unfortunately, the argument that yields the improved rates~\eqref{eq:improved-rates} does not appear amenable to handling more general subspaces of $\mathcal{H}_K(\Omega)$.

By replacing~\eqref{eq:WCE-bounds} with~\eqref{eq:improved-rates} in the proof of Theorem~\ref{thm:coverage-RKHS} we establish that the GP model is asymptotically underconfident for $f \in T(L^2(\Omega))$.

\begin{theorem}[Asymptotic underconfidence for sufficiently regular functions] \label{thm:underconfidence-RKHS} Suppose that \sloppy{${\Omega \subset \R^d}$} is compact, $K$ is continuous, $f \in T(L^2(\Omega)) \subset \mathcal{H}_K(\Omega)$, and there is $x^* \in \Omega$ such that \sloppy{${f(x^*) \neq 0}$} and $x^* \in X_N$ for all sufficiently large $N$.
    Then
    \begin{equation*}
       \sup_{x \in \Omega \setminus X_N} \! \frac{ \abs[0]{ f(x) - s_{f,X_N}(x)}}{ R_\textsm{GP}(x,f,X_N)} \lesssim N^{1/2} \norm[0]{P_{X_N}^{1/2}}_{L^2(\Omega)}  \:\: \text{ and } \:\: \frac{\abs[0]{I(f) - Q_{X_N}(f)}}{ R_\textsm{BC}(f,X_N)} \lesssim N^{1/2} \norm[0]{P_{X_N}^{1/2}}_{L^2(\Omega)}. 
    \end{equation*}
    That is, the model is asympotically underconfident if the sequence \sloppy{${(X_N)_{N=1}^\infty \subset \Omega}$} is such that $N^{1/2} \norm[0]{P_{X_N}^{1/2}}_{L^2(\Omega)} \to 0$ as $N \to \infty$.
\end{theorem}
\begin{proof} 
Let $f = Tv$ for $v \in L^2(\Omega)$. By using~\eqref{eq:improved-rates} instead of~\eqref{eq:WCE-bounds} in the proof of Theorem~\ref{thm:coverage-RKHS} we get
  \begin{equation*}
     \sup_{x \in \Omega \setminus X_N} \! \frac{ \abs[0]{ f(x) - s_{f,X_N}(x)}}{ R_\textsm{GP}(x,f,X_N)} \leq \frac{ \norm[0]{P_{X_N}^{1/2}}_{L^2(\Omega)} \norm[0]{v}_{L^2(\Omega)}}{\sigma_\ML(f,X_N)} \lesssim N^{1/2} \norm[0]{P_{X_N}^{1/2}}_{L^2(\Omega)}. 
  \end{equation*}
  The supremum is over $x \notin X_N$ because for $x \in X_N$ we have defined the standard score to be one.
  The argument for integration is identical.
\end{proof}

If $(X_N)_{N=1}^\infty$ are quasi-uniform (see Section~\ref{sec:Sobolev-rates} for details), then $N^{1/2} \norm[0]{P_{X_N}^{1/2}}_{L^2(\Omega)} \to 0$ is true, for example, when $K$ is one of the popular infinitely smooth kernels associated with super-algebraic rates of convergence such as a Gaussian or an inverse multiquadric~\citep{RiegerZwicknagl2010}. 
A specialisation to Sobolev kernels will be given in Section~\ref{sec:coverage-sobolev}.

\section{Sobolev Kernels and Functions Outside the RKHS} \label{sec:sobolev-kernels}

This section extends the results of Section~\ref{sec:general-kernels} for functions outside the RKHS when the kernel $K$ is a Sobolev kernel.

\subsection{Sobolev Spaces and Kernels} \label{sec:sobolev-spaces}

Let $\widehat{g}(\xi) \defeq \int_{\R^d} g(x) \mathrm{e}^{-\mathrm{i} \b{\xi}^\T \b{x}} \dif \b{x}$ denote the Fourier transform of $g \in L^1(\R^d)$. The Sobolev space $W_2^\alpha(\R^d)$ of order $\alpha \geq 0$ is the Hilbert space
\begin{equation*}
  W_2^\alpha(\R^d) \defeq \Set[\bigg]{g \in L^2(\R^d)}{ \int_{\R^d} \big( 1 + \norm[0]{\b{\xi}}^2 \big)^\alpha \abs[0]{\widehat{g}(\b{\xi})}^2 \dif \b{\xi} < \infty}
\end{equation*}
equipped with the inner product
\begin{equation*}
  \inprod{h}{g}_{W_2^\alpha(\R^d)} \defeq \int_{\R^d} \big( 1 + \norm[0]{\b{\xi}}^2 \big)^\alpha \, \widehat{h}(\b{\xi}) \overline{\widehat{g}(\b{\xi})} \dif \b{\xi},
\end{equation*}
where $\bar{z}$ is the complex conjugate of $z \in \mathbb{C}$.
When $\alpha \in \N$, the space $W_2^\alpha(\R^d)$ can be equivalently defined as consisting of those functions whose weak derivatives up to order $\alpha$ exist and are in $L^2(\R^d)$. For $\alpha \notin \N$, $W_2^\alpha(\R^d)$ can also be defined as an interpolation or Besov space,  up to equivalent norms~\citep[e.g.,][]{Triebel2006}.
If $\alpha > d/2$, then every element of $W_2^\alpha(\mathbb{R}^d)$ can be uniquely identified with a continuous function from its $L^2(\R^d)$ equivalence class and $W_2^\alpha(\mathbb{R}^d)$ can be viewed as an RKHS of continuous functions on $\R^d$.
This identification will be implicitly assumed throughout the article.

Let $\Omega \subset \R^d$ be Lebesgue measurable and let $W_2^\alpha(\Omega)$ be the restriction of $W_2^\alpha(\R^d)$ to $\Omega$, as defined in Section~\ref{subsec: RKHS introduce}.
We say that a kernel $K \colon \Omega \times \Omega \to \R$ is a \emph{Sobolev kernel of order} $\alpha > d/2$ (on $\Omega$) if its RKHS $\mathcal{H}_K(\Omega)$ is \emph{norm-equivalent} to $W_2^\alpha(\Omega)$.
That is, $\mathcal{H}_K(\Omega)$ equals $W_2^\alpha(\Omega)$ as a set of functions and there exist positive constants $C_K$ and $C_K'$ such that
\begin{equation} \label{eq:norm-equivalence}
  C_{K} \norm[0]{g}_{W_2^\alpha(\Omega)} \leq \norm[0]{g}_{\mathcal{H}_K(\Omega)} \leq C_K' \norm[0]{g}_{W_2^\alpha(\Omega)}
\end{equation}
for all $g \in \mathcal{H}_K(\Omega)$.
Stationary kernels with prescribed Fourier decay form an important subclass of Sobolev kernels: if there is $\Phi \colon \R^d \to \R$ such that $K(x,y) = \Phi(x-y)$ and
\begin{equation*}
  C_1 \big(1 + \norm[0]{\xi}^2 \big)^{-\alpha} \leq \widehat{\Phi}(\xi) \leq C_2 \big(1 + \norm[0]{\xi}^2 \big)^{-\alpha} \quad \text{ for some $C_1, C_2 > 0$ and all $\xi \in \R^d$},
\end{equation*}
then $K$ is a Sobolev kernel of order $\alpha$ and $\mathcal{H}_K(\mathbb{R}^d)$ is norm-equivalent to $W_2^\alpha(\mathbb{R}^d)$.
Perhaps the most ubiquitous Sobolev kernels are the Matérn kernels
\begin{equation} \label{eq:matern}
  K_{\nu,\ell}(x, y) = \frac{2^{1-\nu}}{\Gamma(\nu)} \bigg( \frac{\sqrt{2\nu} \norm[0]{x-y}}{\ell} \bigg)^\nu \mathrm{K}_\nu \bigg( \frac{\sqrt{2\nu} \norm[0]{x-y}}{\ell} \bigg),
\end{equation}
where $\nu > 0$ is a smoothness parameter, $\ell > 0$ a length-scale parameter, $\Gamma$ the Gamma function, and $\mathrm{K}_\nu$ the modified Bessel function of the second kind of order $\nu$. The Fourier transform of a Matérn kernel is~\citep[p.\@~49]{Stein1999}
\begin{equation} \label{eq:matern-fourier}
  K_{\nu,\ell}(x,y) = \Phi_{\nu,\ell}(x-y), \quad \quad \widehat{\Phi}_{\nu,\ell}(\xi) = \frac{\Gamma(\nu + d/2)}{ \pi^{d/2} \Gamma(\nu)} \bigg( \frac{2\nu}{\ell^2} \bigg)^\nu \bigg( \frac{2\nu}{\ell^2} + \norm[0]{\xi}^2 \bigg)^{-(\nu+d/2)},
\end{equation}
and its RKHS is thus norm-equivalent to the Sobolev space $W_2^{\alpha}(\R^d)$ with $\alpha = \nu + d/2$. See \citet[Chapter~10]{Wendland2005} for proofs and further detail.

Functions that lie on the ``boundary'' of a Sobolev space play an important role in our analysis. For this purpose, define the sets
\begin{align*}
  S_-^\alpha(\R^d) & \defeq \Set[\big]{ g \in L^2(\R^d)}{ \abs[0]{\widehat{g}(\xi)}^2 \lesssim \big(1 + \norm[0]{\xi}^2 \big)^{-(\alpha + d/2)} }, \\
  S_+^\alpha(\R^d) & \defeq \Set[\big]{ g \in L^2(\R^d)}{ \abs[0]{\widehat{g}(\xi)}^2 \gtrsim \big(1 + \norm[0]{\xi}^2 \big)^{-(\alpha + d/2)} },
\end{align*}
and
\begin{equation} \label{eq:S-definition}
  S^\alpha(\R^d) \defeq S_-^\alpha(\R^d) \cap S_+^\alpha(\R^d).
\end{equation}
From the fact that $\int_{\R^d} ( 1 + \norm[0]{\xi}^2 )^\alpha ( 1 + \norm[0]{\xi}^2 )^{-(\beta + d/2)} \dif \xi$ is finite if and only if $\beta > \alpha$ it follows that $S^\beta(\R^d)$ and $S_-^\beta(\R^d)$ are subsets of $W_2^\alpha(\R^d)$ if and only if $\beta > \alpha$.
Similarly, $S_+^\beta(\R^d) \cap W_2^\alpha(\R^d)$ is non-empty if and only if $\beta > \alpha$, and it may therefore be helpful to think of $S_+^\alpha(\R^d)$ as approximately the collection of square-integrable functions that are not in $W_2^\alpha(\R^d)$.
A function $g \colon \Omega \to \R$ is said to be in $S_-^\alpha(\Omega)$ ($S_+^\alpha(\Omega)$) if it has an extension $g_0 \in S_-^\alpha(\R^d)$ ($g_0 \in S_+^\alpha(\R^d)$).
As an aside, we note the similarity of these sets to the sequence hyperrectangles analysed in \cite{Szabo2013,Szabo2015} and \cite{HadjiSzabo2019}.

\subsection{Motivation: Sample Path Properties of Gaussian Processes} \label{sec:sample-paths}

In this article the function $f$ is fixed, but nevertheless it seems reasonable that a statistical estimation method based on a GP model ought to perform well when the assumptions of the GP model are satisfied. 
This motivates us to consider the regularity of samples from the GP model, which will later form the basis of regularity assumptions on $f$.
The most important results relating the samples and the RKHS are the following~\citep[for a recent review, see][]{Kanagawa2018}:
\begin{itemize}
\item If $\mathcal{H}_K(\Omega)$ is infinite-dimensional, then the sample paths of the GP belong to $\mathcal{H}_K(\Omega)$ with probability 0. In general, the samples being contained in the RKHS of a different kernel $R$  with probability 0 or 1 depends on whether or not a certain nuclear dominance condition between the kernels $K$ and $R$ holds~\citep{Driscoll1973,LukicBeder2001}.
\item If $K$ is a Sobolev kernel of order $\alpha > d/2$, then the GP sample paths are in $W_2^\beta(\Omega)$ with probability 1 if $\beta < \alpha - d/2$ and with probability 0 if $\beta \geq \alpha - d/2$~\citep{Scheuerer2010,Steinwart2017}.
\end{itemize}
The latter result essentially says that for Sobolev kernels the sample paths are rougher than elements in $\mathcal{H}_K(\Omega)$ by order $d/2$. Furthermore, it follows that sample paths are in the set
\begin{equation} \label{eq:sample-path-epsilon}
  W_2^{\alpha-d/2-\varepsilon}(\Omega) \setminus W_2^{\alpha-d/2}(\Omega)
\end{equation}
with probability 1 for any $\varepsilon > 0$.
We are not aware of more advanced developments than this but, encouraged by~\eqref{eq:sigma-ML-samples}, conjecture that the set $S^{\alpha-d/2}(\Omega)$, which is a subset of~\eqref{eq:sample-path-epsilon} for any $\varepsilon > 0$, is in some sense the smallest set (or closely related to such a set) that contains almost all sample paths of a GP with a Sobolev covariance kernel of order $\alpha$.

\subsection{Error Estimates for Sobolev Kernels} \label{sec:Sobolev-rates}

In this section we present bounds on the GP approximation and integration errors and sharp rates (i.e., the upper and lower bounds are of matching order) of decay of $\sup_{ \b{x} \in \Omega} P_{X}(\b{x},\b{x})^{1/2}$ and $V_X^{1/2}$ when $K$ is a Sobolev kernel; these will be used to study the maximum likelihood estimator in Section~\ref{sec:MLE-Sobolev}.
Define the \emph{fill-distance} $h_X$ and the \emph{separation radius} $q_X$ of a set of distinct points $X = \{x_1, \ldots, x_N\} \subset \Omega$ as 
\begin{equation*}
  h_X \defeq \sup_{ x \in \Omega} \, \min_{i=1,\ldots,N} \, \norm[0]{ x - x_i} \quad \text{ and } \quad q_X \defeq \frac{1}{2} \min_{ i \neq j} \, \norm[0]{x_i - x_j}.
\end{equation*}
Also define the \emph{mesh ratio} $\rho_X \defeq h_X / q_X \geq 1$. 
A sequence $(X_N)_{N=1}^\infty \subset \Omega$ is \emph{quasi-uniform} if $\rho_{X_N} \lesssim 1$, which implies that $q_{X_N} \asymp h_{X_N} \asymp N^{-1/d}$~\citep[Proposition~14.1]{Wendland2005}.

The domain $\Omega \subset \R^d$ will often be assumed to satisfy the following requirement, which will be made explicit when required.

\begin{assumption} \label{ass:Omega} The set $\Omega \subset \R^d$ is bounded and connected, has a non-empty interior and a Lipschitz boundary, and satisfies an interior cone condition.
\end{assumption}

The Lipschitz boundary condition says that the boundary is sufficiently regular in that it is locally the graph of a Lipschitz function, while the interior cone condition prohibits the existence of pinch points; for technical definitions see for example~\citet[Section~3]{Kanagawa2019}. These conditions are standard in the theory of Sobolev spaces and error analysis of kernel-based approximation methods.\footnote{In the results we cite it is often assumed that $\Omega$ is open. Because these results provide bounds on $L^p(\Omega)$ norms and a Lipschitz boundary is of measure zero, they remain valid whenever $\Omega$ has a non-empty interior.}
In particular, they guarantee that various different notions of integer and fractional order Sobolev spaces defined on $\Omega$ result in identical function spaces up to equivalent norms.
Assumption~\ref{ass:Omega} is satisfied by all typical domains and in particular by $\Omega = [0,1]^d$, which is used in the numerical examples in Section~\ref{sec:examples}.

The following theorem provides bounds on the approximation and integration error by a GP conditional mean when the kernel is Sobolev and $f$ does not necessarily lie in the RKHS. 
The theorem as we state it is a consequence of results in the scattered data approximation literature~\citep{WendlandRieger2005,Narcowich2006}.
For completeness and to simplify later developments the proof is provided in Appendix~\ref{sec:proofs}.

\begin{theorem} \label{thm:sobolev-rates} Let $\alpha \geq \beta$ and $\floor{\beta} > d/2$. Suppose that $\Omega \subset \R^d$ satisfies Assumption~\ref{ass:Omega} and $K$ is a Sobolev kernel of order $\alpha$. If $f \in W_2^\beta(\Omega)$, then there are $C_1, C_2, h_0 > 0$, which do not depend on $f$ or $X$, such that
  \begin{equation*}
    \sup_{ x \in \Omega} \, \abs[0]{ f(x) - s_{f,X}(x) } \leq C_1 h_X^{\beta-d/2} \rho_X^{\alpha-\beta} \norm[0]{f}_{W_2^\beta(\Omega)} \: \text{ and } \: \abs[0]{I(f) - Q_X(f)} \leq C_2 h_X^\beta \rho_X^{\alpha-\beta} \norm[0]{f}_{W_2^\beta(\Omega)}
  \end{equation*}
  whenever $h_X \leq h_0$.
  For a quasi-uniform sequence $(X_N)_{N=1}^\infty \subset \Omega$ these bounds become
  \begin{equation*}
    \sup_{x \in \Omega} \, \abs[0]{ f(x) - s_{f,X_N}(x) } \lesssim N^{-\beta/d+1/2} \norm[0]{f}_{W_2^\beta(\Omega)} \quad \text{ and } \quad \abs[0]{I(f) - Q_{X_N}(f)} \lesssim N^{-\beta/d} \norm[0]{f}_{W_2^\beta(\Omega)}.
  \end{equation*}
\end{theorem}

See \citet{Arcangeli2007,Arcangeli2012} and \citet{Wynne2019} for a collection of marginally more general versions of Theorem~\ref{thm:sobolev-rates}.
  These generalisations are not used here because proofs of some of the results in Section~\ref{sec:MLE-Sobolev} require understanding of the dependency, which is much less transparent in the generalisations, on the Sobolev smoothness parameters of the constants~$C_1$ and $C_2$.
The following extension for $f \in S_-^\beta(\Omega)$, that we have not found in the literature, will be useful. Its proof is given in Appendix~\ref{sec:proofs}.

\begin{theorem} \label{thm:sobolev-rates-S} Suppose that the other assumptions of Theorem~\ref{thm:sobolev-rates} are satisfied but $f \in S_-^\beta(\Omega)$. Then for a quasi-uniform sequence $(X_N)_{N=1}^\infty \subset \Omega$,
\begin{equation*}
  \sup_{x \in \Omega} \, \abs[0]{ f(x) - s_{f,X_N}(x) } \lesssim N^{-\beta/d+1/2} (\log N)^{1/2} \:\: \text{ and } \:\: \abs[0]{I(f) - Q_{X_N}(f)} \lesssim N^{-\beta/d} (\log N)^{1/2}.
\end{equation*}
\end{theorem}

The implicit constants in the estimates of Theorem~\ref{thm:sobolev-rates-S} depend on the smallest $C > 0$ such that \sloppy{${\abs[0]{\widehat{f_0}(\xi)}^2 \leq C(1+\norm[0]{\xi}^2)^{-(\beta+d/2)}}$} for a Sobolev exntesion $f_0$ of $f$ and all sufficiently large $\xi \in \R^d$. Similar dependencies are implicit in the bounds of Propositions~\ref{thm:MLE-upper-S} and~\ref{thm:MLE-lower-S} and Theorems~\ref{cor:MLE-upper-lower} and~\ref{thm:coverage-sobolev}.

Due to~\eqref{eq:WCE} the error estimates of Theorem~\ref{thm:sobolev-rates} for $\beta = \alpha$ are also upper bounds on the conditional standard deviations.
It is possible to establish matching lower bounds, which leads to the following standard result, the proof of which is given in Appendix~\ref{sec:proofs}.

\begin{theorem} \label{thm:WCE-exact} Suppose that $\Omega \subset \R^d$ satisfies Assumption~\ref{ass:Omega}. If $K$ is a Sobolev kernel of order $\alpha > \floor{d/2}$ and the sequence $(X_N)_{N=1}^\infty \subset \Omega$ is quasi-uniform, then
  \begin{equation*}
     \sup_{x \in \Omega} P_{X_N}(x,x)^{1/2} \asymp N^{-\alpha/d + 1/2} \quad \text{ and } \quad V_{X_N}^{1/2} \asymp N^{-\alpha/d}.
  \end{equation*}
  Furthermore, $P_{X_N}(\b{x}, \b{x})^{1/2} \asymp N^{-\alpha/d + 1/2}$ for any $x \notin \bigcup_{N=1}^\infty X_N$.
\end{theorem}

\subsection{Maximum Likelihood Estimation} \label{sec:MLE-Sobolev}

This section contains upper and lower bounds on $\sigma_\ML(f,X_N)$ when $K$ is a Sobolev kernel of order $\alpha$ and $f$ is not necessarily in $W_2^\alpha(\Omega)$.
If $f \in W_2^\alpha(\Omega)$, which below corresponds to either $\beta \geq \alpha$ or $\beta < \alpha$, then the results in Section~\ref{subsec: MLEs in RKHS} can be used instead.
The main result on maximum likelihood estimation is Theorem~\ref{cor:MLE-upper-lower} which provides sharp (up to logarithmic factors) asymptotics for the maximum likelihood estimate under certain conditions on $f$.
Propositions~\ref{thm:MLE-upper} to~\ref{thm:MLE-lower-S} contain individual upper and lower bounds.
The bounds are used to discuss credible sets and asymptotic overconfidence and underconfidence in Section~\ref{sec:coverage-sobolev}.
The proofs of this section are provided in Appendix~\ref{sec:proofs}.

\begin{proposition} \label{thm:MLE-upper} Let $\alpha \geq \beta$ and $\floor{\beta} > d/2$. Suppose that $\Omega \subset \R^d$ satisfies Assumption~\ref{ass:Omega} and $K$ is a Sobolev kernel of order $\alpha$. If $f \in W_2^\beta(\Omega)$, then there are $C, h_0 > 0$, which do not depend on $f$ or $X_N$, such that
  \begin{equation} \label{eq:MLE-upper-general}
    \sigma_\ML(f,X_N) \leq C N^{-1/2} q_{X_N}^{\beta-\alpha} \norm[0]{f}_{W_2^\beta(\Omega)}
  \end{equation}
  whenever $h_{X_N} \leq h_0$.
For a quasi-uniform sequence $(X_N)_{N=1}^\infty \subset \Omega$ this bound becomes
  \begin{equation*}
    \sigma_\ML(f,X_N) \lesssim N^{(\alpha-\beta)/d-1/2} \norm[0]{f}_{W_2^\beta(\Omega)}.
  \end{equation*}
\end{proposition}

Proposition~\ref{thm:MLE-upper} holds in a slightly modified form if $f \in S_-^\beta(\Omega)$ (recall $S_-^\beta(\Omega) \cap W_2^\beta(\Omega) = \emptyset)$.

\begin{proposition} \label{thm:MLE-upper-S} 
Suppose that the other assumptions of Theorem~\ref{thm:MLE-upper} are satisfied but \sloppy{${f \in S_-^\beta(\R^d)}$}. Then for a quasi-uniform sequence $(X_N)_{N=1}^\infty \subset \Omega$,
  \begin{equation*}
    \sigma_\ML(f,X_N) \lesssim N^{(\alpha-\beta)/d-1/2} (\log N)^{1/2}.
  \end{equation*}
\end{proposition}

Lower bounds require some additional assumptions and take a more cumbersome form.
Recall that the support of a function is the closed set $\supp(f) \coloneqq \overline{\Set{ x \in \Omega }{ f(x) \neq 0}}$ and the interior $\inte (\Omega)$ of $\Omega \subset \R^d$ is the largest open set contained in $\Omega$.

\begin{proposition} \label{thm:MLE-lower} Let $\alpha \geq \beta > \gamma$ and $\floor{\gamma} > d/2$. Suppose that $\Omega \subset \R^d$ satisfies Assumption~\ref{ass:Omega} and $K$ is a Sobolev kernel of order $\alpha$. If $\supp(f) \subset \inte(\Omega)$ and $f$ has an extension $f_0 \in W_2^\gamma(\R^d) \cap S_+^\beta(\R^d)$ such that $\supp(f_0) \subset \inte(\Omega)$, then there are $C, h_0 > 0$, which do not depend on $X_N$, such that
  \begin{equation} \label{eq:MLE-lower-general}
    \sigma_\ML(f,X_N) \geq C N^{-1/2} h_{X_N}^{\gamma(1-\alpha/\beta)} \rho_{X_N}^{-(\alpha-\gamma)(\alpha-\beta)/\beta} \norm[0]{f}_{W^\gamma_2(\Omega)}^{1-\alpha/\beta}
  \end{equation}
  whenever $h_{X_N} \leq h_0$.
For a quasi-uniform sequence $(X_N)_{N=1}^\infty \subset \Omega$ this bound becomes
  \begin{equation} \label{eq:MLE-lower-uniform}
    \sigma_\ML(f,X_N) \gtrsim N^{\gamma(\alpha/\beta-1)/d - 1/2}\norm[0]{f}_{W^{\gamma}_{2}(\Omega)}^{1-\alpha/\beta}.
  \end{equation}
\end{proposition}

\begin{proposition} \label{thm:MLE-lower-S} 
Suppose that the other assumptions of Theorem~\ref{thm:MLE-lower} are satisfied but \sloppy{${f_0 \in S^\beta(\R^d)}$}. 
Then for a quasi-uniform sequence $(X_N)_{N=1}^\infty \subset \Omega$,
  \begin{equation*}
    \sigma_\ML(f,X_N) \gtrsim N^{(\alpha-\beta)/d-1/2} (\log N)^{(1-\alpha/\beta)/2} .
  \end{equation*}
\end{proposition}

By combining Propositions~\ref{thm:MLE-upper-S} and~\ref{thm:MLE-lower-S} we obtain a rate for $f \in S^\beta(\Omega)$ that is sharp up to logarithmic factors. 
The empirical results in Section~\ref{sec:examples-mle} suggest that elimination of the logarithmic factors and the support conditions may be possible with more careful analysis.

\begin{theorem}[Asymptotics of the MLE] \label{cor:MLE-upper-lower} Let $\alpha \geq \beta$ and $\floor{\beta} > d/2$. Suppose that $\Omega \subset \R^d$ satisfies Assumption~\ref{ass:Omega} and $K$ is a Sobolev kernel of order $\alpha$.
  If $\supp(f) \subset \inte(\Omega)$ and $f$ has an extension $f_0 \in S^\beta(\R^d)$ such that $\supp(f_0) \subset \inte(\Omega)$, then for a quasi-uniform sequence $(X_N)_{N=1}^\infty \subset \Omega$,
  \begin{equation*}
    N^{(\alpha-\beta)/d-1/2} (\log N )^{(1-\alpha/\beta)/2} \lesssim \sigma_\ML(f,X_N) \lesssim N^{(\alpha-\beta)/d-1/2} (\log N)^{1/2}.
  \end{equation*}
\end{theorem}

In particular, for $\beta = \alpha - d/2$ we have $(\alpha-\beta)/d - 1/2 = 0$ so that the maximum likelihood estimates are asymptotically constant, up to logarithmic factors:
\begin{equation} \label{eq:sigma-ML-samples}
  (\log N)^{-d/(4\alpha - 2d)} \lesssim \sigma_\ML(f, X_N) \lesssim (\log N)^{1/2}
\end{equation}
if $f \in S^{\alpha-d/2}(\Omega)$.
As discussed in Section~\ref{sec:sample-paths}, this corresponds to the case where $f$ has essentially the same regularity as samples from a GP whose covariance kernel is a Sobolev kernel of order $\alpha$.

\subsection{Credible Sets} \label{sec:coverage-sobolev}

We now use the bounds on the maximum likelihood estimates to prove an overconfidence result similar to Theorem~\ref{thm:coverage-RKHS}, but this time for functions outside the RKHS.
First, it is instructive to study what can happen if the scale parameter is held fixed. If $K$ is a Sobolev kernel of order $\alpha$, $f \in W_2^\beta(\Omega)$ for $\beta \leq \alpha$, and $(X_N)_{N=1}^\infty$ are quasi-uniform, then Theorems~\ref{thm:sobolev-rates} and~\ref{thm:WCE-exact} yield
\begin{equation} \label{eq:overconfidence-fixed-sigma}
  \frac{\abs[0]{I(f) - Q_{X_N}(f)}}{R_\BC(f, X_N)} = \frac{\abs[0]{I(f) - Q_{X_N}(f)}}{\sigma V_{X_N}^{1/2}} \lesssim \frac{N^{-\beta/d} \norm[0]{f}_{W_2^\beta(\Omega)}}{\sigma N^{-\alpha/d}} = \sigma^{-1} N^{(\alpha-\beta)/d} \norm[0]{f}_{W_2^\beta(\Omega)}.
\end{equation}
That is, there is potential for significant overconfidence if $K$ is smoother than $f$.
The following theorem shows that maximum likelihood estimation provides protection against such model misspecification.

\begin{theorem}[Slow overconfidence at worst outside the RKHS] \label{thm:coverage-sobolev} Let $\alpha \geq \beta$ and $\floor{\beta} > d/2$. Suppose that $\Omega \subset \R^d$ satisfies Assumption~\ref{ass:Omega} and $K$ is a Sobolev kernel of order $\alpha$.
  If $\supp(f) \subset \inte (\Omega)$ and $f$ has an extension $f_0 \in S^\beta(\R^d)$ such that $\supp(f_0) \subset \inte (\Omega)$, then for a quasi-uniform sequence $(X_N)_{N=1}^\infty \subset \Omega$,
  \begin{equation*}
    \frac{\abs[0]{ f(x) - s_{f,X_N}(x)}}{ R_\GP(x,f,X_N)} \lesssim N^{1/2} (\log N)^{\alpha/(2\beta)} \quad \text{ for any } \quad x \in \Omega
  \end{equation*}
  and
  \begin{equation*}
    \frac{\abs[0]{ I(f) - Q_{X_N}(f)}}{ R_\BC(f,X_N)} \lesssim N^{1/2} (\log N)^{\alpha/(2\beta)}.
  \end{equation*}
\end{theorem} 
\begin{proof}
  Consider first approximation with GPs.
  For $N$ such that $x \in X_N$ the standard score in~\eqref{eq:overconfident} and~\eqref{eq:underconfident} is by definition equal to one.
  We can thus assume that $x \notin \bigcup_{N=1}^\infty X_N$.
  Then the estimates in Theorems~\ref{thm:sobolev-rates-S} and~\ref{thm:WCE-exact} and Proposition~\ref{thm:MLE-lower-S} yield
  \begin{equation*}
    \begin{split}
      \frac{\abs[0]{ f(x) - s_{f,X_N}(x)}}{ R_\GP(x,f,X_N)} = \frac{\abs[0]{f(x) - s_{f,X_N}(x)}}{\sigma_\ML(f,X_N) P_{X_N}(x,x)^{1/2} } & \lesssim \frac{ N^{-\beta/d + 1/2} (\log N)^{1/2}}{ N^{(\alpha-\beta)/d - 1/2} (\log N)^{(1-\alpha/\beta)/2} N^{-\alpha/d + 1/2}} \\
      &= N^{1/2} (\log N)^{\alpha/(2\beta)}.
      \end{split}
  \end{equation*}
  The proof for integration is essentially identical.
\end{proof}

Interestingly, the case $\beta = \alpha - d/2$, which essentially corresponds to $f$ having the same regularity as samples from the GP, plays no special role in Theorem~\ref{thm:coverage-sobolev}.
We are uncertain if this is due to an inadequacy in the analysis or if there in fact exist GP samples for which the model is overconfident.
In practice one rarely knows the exact smoothness of $f$ (or the function is not an element of $S^\beta(\Omega)$ for any $\beta$) and can only guess, for example, that $f$ has weak derivatives at least up to some order $\beta$. If $\beta < \alpha$, then nothing can be inferred about the credible sets based on our results; if $\beta \geq \alpha$, then Theorem~\ref{thm:coverage-RKHS} can be used.

As our final result we present a specialisation of Theorem~\ref{thm:underconfidence-RKHS} to Sobolev kernels. The proof is a straightforward application of the estimates in~\eqref{eq:improved-rates}, Proposition~\ref{prop:MLE-general}, and Theorem~\ref{thm:WCE-exact}:
\begin{equation*}
 \sup_{x \in \Omega \setminus X_N} \! \frac{ \abs[0]{ f(x) - s_{f,X_N}(x)}}{ R_\textsm{GP}(x,f,X_N)} \lesssim N^{1/2} \bigg( \int_\Omega P_{X_N}(x,x) \dif x \bigg)^{1/2} \leq N^{1/2} \sup_{x \in \Omega} P_{X_N}(x,x)^{1/2} \asymp N^{-\alpha/d + 1} 
\end{equation*}
if the point sequence is quasi-uniform.
Recall that $T(L^2(\Omega))$ is the range of the integral operator in~\eqref{eq:integral-operator}. 

\begin{theorem}[Asymptotic underconfidence for sufficiently regular functions] \label{thm:underconfidence-sobolev} Suppose that $\Omega \subset \R^d$ is compact, $K$ is a Sobolev kernel of order $\floor{\alpha} > d/2$, and $f \in T(L^2(\Omega))$. Then for a quasi-uniform sequence $(X_N)_{N=1}^\infty \subset \Omega$ such that there is $x^* \in X_N$ for which $f(x^*) \neq 0$ for all sufficiently large $N$, 
  \begin{equation*}
     \sup_{x \in \Omega \setminus X_N} \! \frac{ \abs[0]{ f(x) - s_{f,X_N}(x)}}{ R_\textsm{GP}(x,f,X_N)} \lesssim N^{-\alpha/d + 1}  \quad \text{ and } \quad \frac{\abs[0]{I(f) - Q_{X_N}(f)}}{ R_\textsm{BC}(f,X_N)} \lesssim N^{-\alpha/d + 1}. 
  \end{equation*}
\end{theorem}

We thus have asymptotic underconfidence for approximation and integration of $f \in T(L^2(\Omega))$ at least when $\alpha > d$.
Note that for Sobolev kernels of order $\alpha$ the range $T(L^2(\Omega))$ is related to the Sobolev space of smoothness $2\alpha$~\citep[see, e.g.,][Section~2.3]{TuoWangWu2019}.

Theorem~\ref{thm:underconfidence-sobolev} can be illustrated in detail using the Brownian motion kernel \sloppy{${K(x,y) = \min\{x,y\}}$} on $\Omega = [0,1]$.
Its RKHS consists of functions $f \in W_2^1([0,1])$ such that $f(0) = 0$.
It is well-known that the GP conditional mean for this kernel is the piecewise linear spline interpolant. Furthermore, for the weight $w \equiv 1$ in~\eqref{eq:integral} the Bayesian quadrature estimator is the \emph{trapezoidal rule} if $x_N = 1$ and $f(0) = 0$~\citep[e.g.,][Section~5.5]{Karvonen2019}:
\begin{equation*}
  Q_{X_{N}}(f) = \sum_{n=1}^N \frac{f(x_{n-1}) + f(x_n)}{2} (x_n - x_{n-1}),
\end{equation*}
where the convention $x_0 = 0$ is used.
If the equispaced points $x_n = n/N$ are used, the integral conditional variance~\eqref{eq:BQ-var} has the simple form~\citep[p.\@~26]{Ritter2000}
\begin{equation} \label{eq:trapezoidal-wce}
  V_{X_{N}} = \frac{1}{12} \sum_{n=1}^N (x_n - x_{n-1})^3 = \frac{1}{12 N^2}.
\end{equation}
If $f \colon [0,1] \to \R$ is twice-differentiable with $f''$ bounded and $f(0) = 0$, which means that $f \in W_2^2([0,1])$, the standard error formula for the trapezoidal rule with equispaced points is~\citep[Section~5.1]{Atkinson1989}
\begin{equation} \label{eq:trapezoidal-error}
  I(f) - Q_{X_N}(f) = -\frac{1}{12N^2} f''(\xi_N) \quad \text{ for some } \quad \xi_N \in [0,1].
\end{equation}
Because such a function is in the RKHS, Proposition~\ref{prop:MLE-general} gives $\sigma_\ML(f,X_N) \asymp N^{-1/2}$.
This and the estimates~\eqref{eq:trapezoidal-wce} and~\eqref{eq:trapezoidal-error} thus yield
\begin{equation} \label{eq:trapezoidal-coverage}
  \frac{\abs[0]{I(f) - Q_{X_N}(f)}}{R_\BC(f,X_N)} = \frac{\abs[0]{I(f) - Q_{X_N}(f)}}{\sigma_\ML(f,X_N) V_{X_N}^{1/2}} \leq \frac{\sup_{ x \in [0,1] } \abs[0]{f''(x)}}{\sqrt{12} N \sigma_\ML(f,X_N)} \lesssim N^{-1/2},
\end{equation}
which is the statement of Theorem~\ref{thm:underconfidence-sobolev} with $\alpha = 1$ and $d = 1$.
If we further assume that there is $C > 0$ such that $f''(x) > C$ for all $x \in [0,1]$ (i.e., $f$ is strictly convex), then~\eqref{eq:trapezoidal-error} implies that $\abs[0]{I(f) - Q_{X_N}(f)} \asymp N^{-2}$, and the standard score~\eqref{eq:trapezoidal-coverage} hence has a lower bound of matching order $N^{-1/2}$.


\section{Numerical Illustration} \label{sec:examples}

This section numerically investigates the sharpness of the results in Section~\ref{sec:sobolev-kernels}. 
Examples in Section~\ref{sec:examples-mle} verify that the bounds on $\sigma_\ML(f,X_N)$ in Corollary~\ref{cor:MLE-upper-lower} are valid. 
Section~\ref{sec:example-coverage} contains limited evidence that the bounds in Section~\ref{sec:coverage-sobolev} are not tight: the credible sets do \emph{not} appear to contract with a rate $O(N^{-1/2})$ faster than the true error.

\subsection{Maximum Likelihood Estimation} \label{sec:examples-mle}

In these examples we illustrate the behaviour of the maximum likelihood estimate $\sigma_\ML(f,X_N)$ using the Matérn kernel $K_{\nu,\ell}$ in~\eqref{eq:matern}. Recall that the RKHS of $K_{\nu,\ell}$ is norm-equivalent to the Sobolev space $W_2^\alpha(\R^d)$ with \sloppy{${\alpha = \nu + d/2}$}.
We select $\Omega = [0,1]^d$ and use test functions constructed out of Matérn kernels of smoothness $\eta$:
\begin{equation} \label{eq:matern-test-function}
  f(\b{x}) = \sum_{i=1}^m a_i K_{\eta,\ell}(\b{x},\b{z}_i), \quad \text{ with some } \quad a_i \in \R \quad \text{ and } \quad \b{z}_i \in [0,1]^d.
\end{equation}
By~\eqref{eq:matern-fourier}, the Fourier transform of such a function satisfies $\abs[0]{\widehat{f}(\b{\xi})}^2 \propto (2\eta/\ell^2 + \norm[0]{\xi}^2)^{-(2\eta+d)}$.
The function is thus an element $S^{2\eta+d/2}(\R^d)$ and, except for the support condition $\supp(f) \subset (0,1)^d$, satisfies the assumptions of Corollary~\ref{cor:MLE-upper-lower} with $\beta = 2\eta + d/2$.
For a quasi-uniform point sequence we therefore expect that (possibly up to logarithmic factors)
\begin{equation} \label{eq:theoretical-rate}
  \sigma_\ML(f,X_N) \asymp N^{(\nu-2\eta)_+/d - 1/2} \: \text{ if } \: \nu \geq 2\eta \quad \text{ and } \quad \sigma_\ML(f,X_N) \asymp N^{-1/2} \: \text{ if } \: \nu < 2\eta.
\end{equation}

\begin{figure}[t!]
  \centering
  \includegraphics{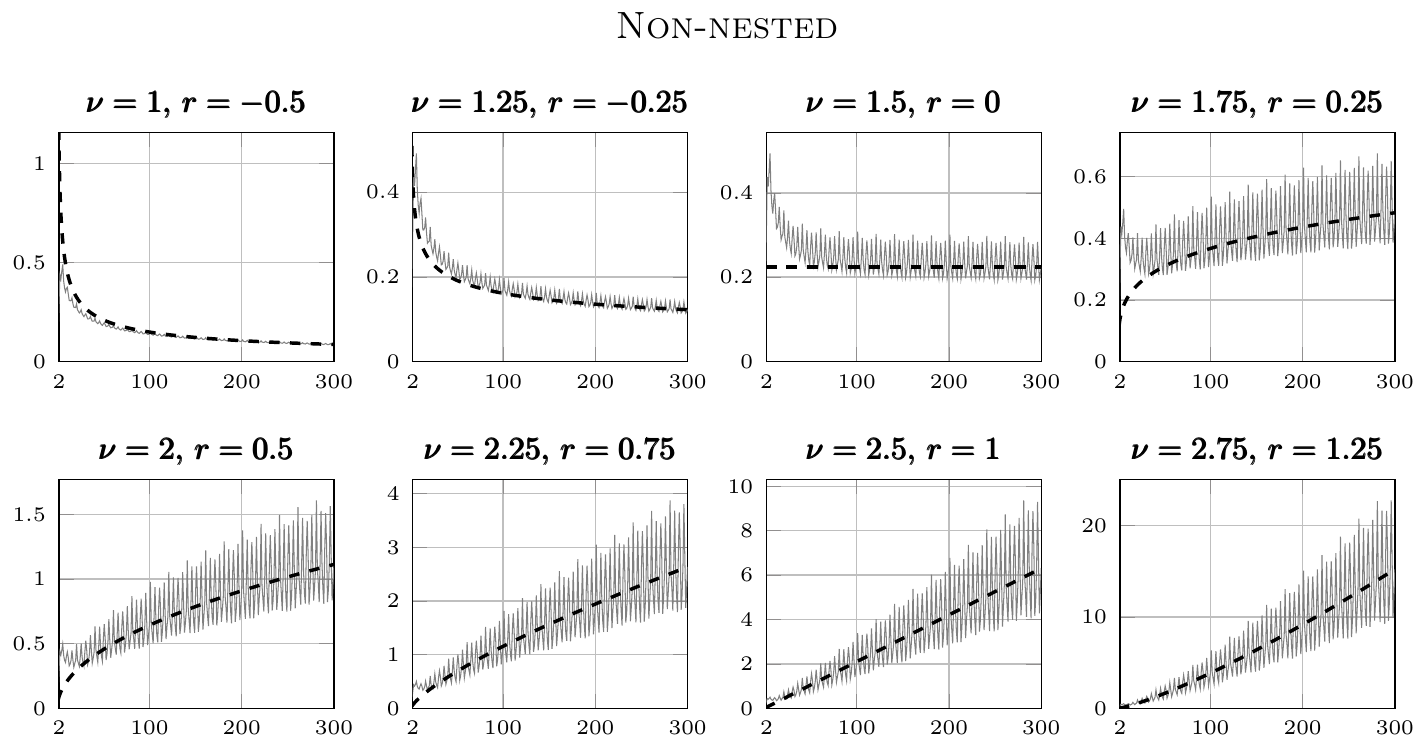}
  \includegraphics{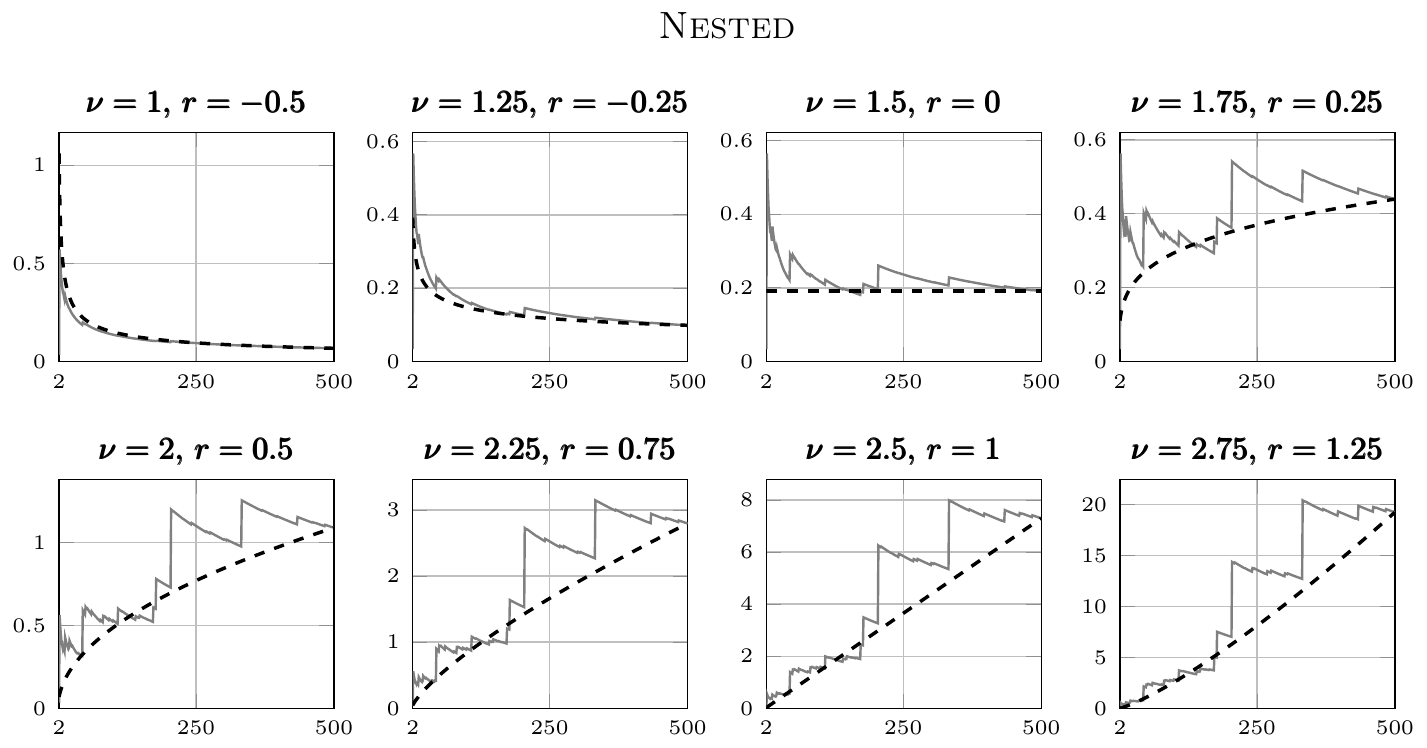}
  \caption{Maximum likelihood estimates $\sigma_\ML(f,X_N)$ (grey) and theoretically predicted rates $N^r$ with $r = \nu-3/2$ (dashed black) when $d=1$, as a function of the size $N$ of the point set. \emph{Above:} non-nested uniform point sets~\eqref{eq:XN-uniform} for $N=2,\ldots,300$. \emph{Below:} nested van der Corput points~\eqref{eq:van-der-corput} for $N=2,\ldots,500$. The function $f$ is of the form~\eqref{eq:matern-test-function} with $\eta = 0.5$. The GP covariance kernel is a Matérn~\eqref{eq:matern} with smoothness~$\nu$.} \label{fig:MLE-1D}
\end{figure}

\paragraph{MLE when $d=1$} In the first example we set $d = 1$, $\ell=0.2$, $\eta=0.5$, $m=3$, $(a_1,a_2,a_3) = (1,0.5,0.2)$, and $(z_1, z_2, z_3) = (0.2, 0.55, 0.78)$.
Figure~\ref{fig:MLE-1D} displays the behaviour of the maximum likelihood estimates and the predicted theoretical rates~\eqref{eq:theoretical-rate} for different values of the smoothness parameter $\nu$ of the Matérn kernel. On the first and second row of Figure~\ref{fig:MLE-1D} the point sets are the non-nested uniform grids
\begin{equation} \label{eq:XN-uniform}
  X_N = \bigg\{ 0, \frac{1}{N-1}, \ldots, \frac{N-2}{N-1}, 1 \bigg\}
\end{equation}
with fill-distances $h_{X_N} = 1/(N-1)$.
The maximum likelihood estimates exhibit wild oscillations which seem to be related to placement of the evaluation points in relation to the points $z_i$ defining $f$. Nevertheless, it is clear that the rates predicted by~\eqref{eq:theoretical-rate} are realised in all cases. On the third and fourth row of Figure~\ref{fig:MLE-1D} the point sets are nested: $X_N$ consists of the first $N$ elements of the low-discrepancy van der Corput sequence
\begin{equation} \label{eq:van-der-corput}
  0, \: 0.5, \: 0.75, \: 0.125, \: 0.625, \: 0.375, \: 0.875, \: \ldots
\end{equation}
Because the fill-distances and separation radii of these sets are not equal, the maximum likelihood estimates exhibit sudden increases intercepted by periods of decay contributed by the $N^{-1/2}$ term in~\eqref{eq:MLE-upper-general} and~\eqref{eq:MLE-lower-general}.
However, overall behaviour of $\sigma_\ML(f,X_N)$ appears to be compatible with the rate~\eqref{eq:theoretical-rate}.
Even though the plots are seemingly similar, $\sigma_\ML(f,X_N)$ grows much faster for larger $\nu$ as attested by changing $y$-scaling of the figures.

\paragraph{MLE when $d=2$} In the second example we set $d = 2$, $\ell = 0.8$, $\eta = 0.75$, $m=3$, $(a_1,a_2,a_3) = (1,0.5,0.2)$, and $(z_1,z_2,z_3) = ((0.1,0.1), (0.5,0.1), (0.725,0.565))$.
The results are displayed in Figure~\ref{fig:MLE-2d}.
The point sets are now Cartesian products of the point sets used in the previous one-dimensional example.
The maximum likelihood estimates again appear to behave as predicted by~\eqref{eq:theoretical-rate}.

\begin{figure}[t!]
  \centering
  \includegraphics{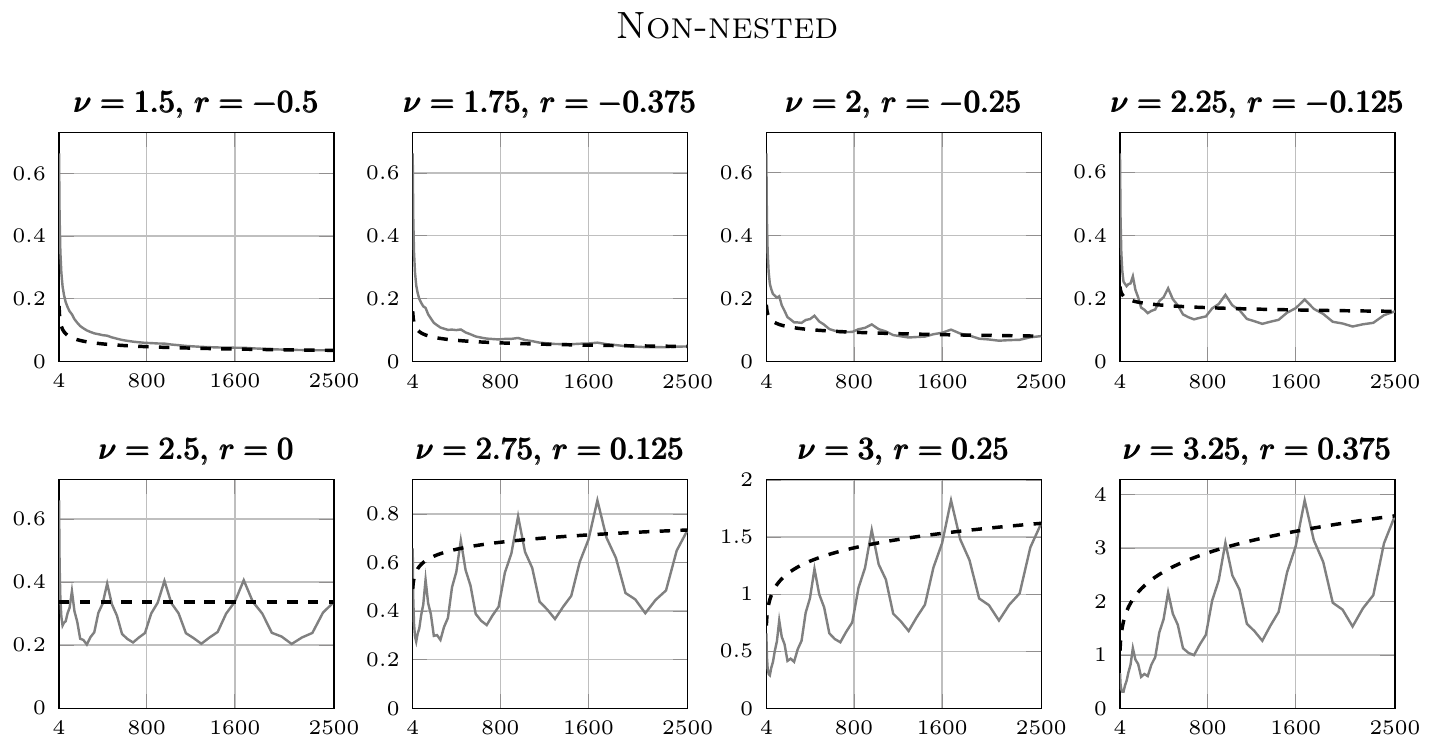}
  \includegraphics{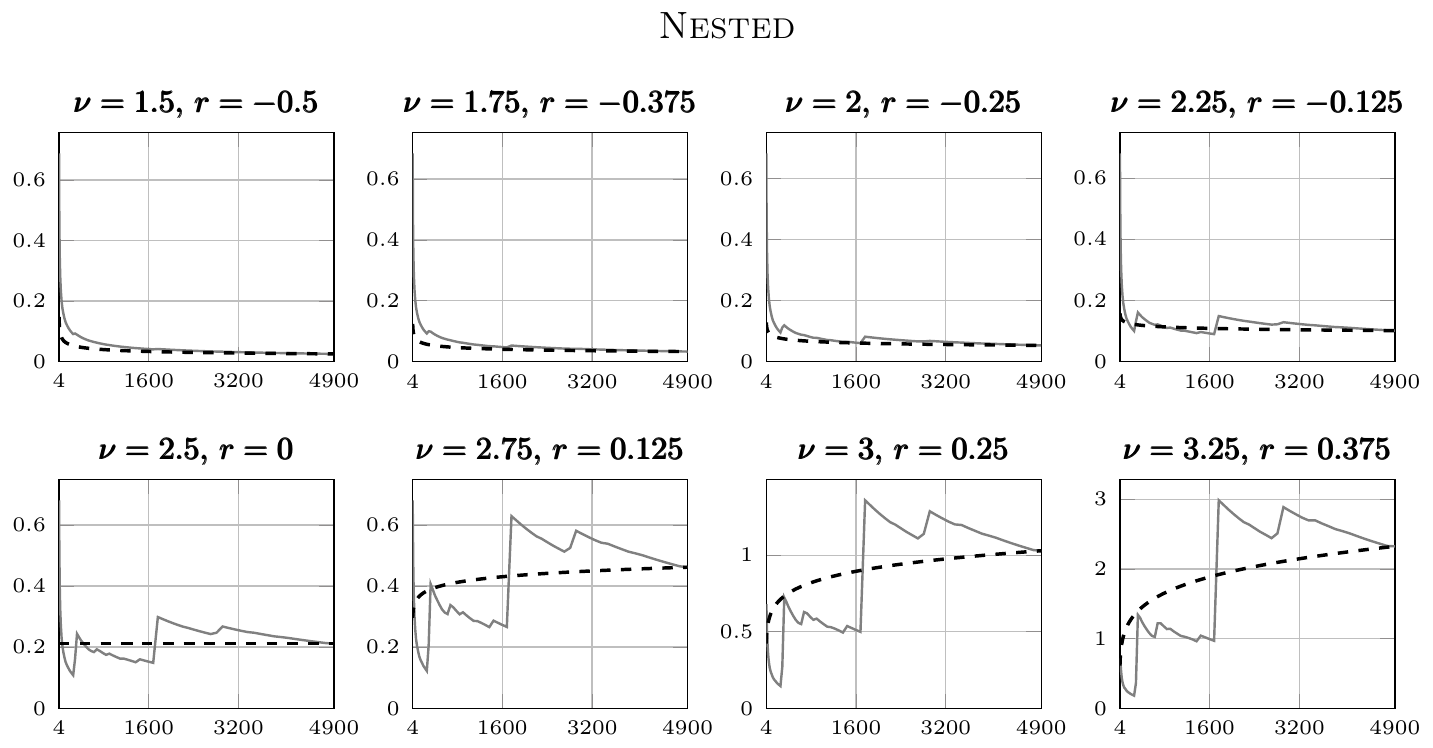}
  \caption{Maximum likelihood estimates $\sigma_\ML(f,X_N)$ (grey) and theoretically predicted rates $N^r$ with $r = \nu/2 - 5/4$ (dashed black) when $d=2$, as a function of the size $N$ of the point set. \emph{Above:} Cartesian products of non-nested uniform point sets~\eqref{eq:XN-uniform} for $N=2^2,\ldots,50^2$. \emph{Below:} nested van der Corput points~\eqref{eq:van-der-corput} for $N=2^2,\ldots,70^2$. The function $f$ is of the form~\eqref{eq:matern-test-function} with $\eta = 0.75$. The GP covariance kernel is a Matérn~\eqref{eq:matern} with smoothness~$\nu$.} \label{fig:MLE-2d}
\end{figure}

\subsection{Credible Sets} \label{sec:example-coverage}

The uncertainty quantification provided by GPs, as measured by the true function or its integral being contained in the credible sets~\eqref{eq:credible-sets}, has been empirically assessed by various authors in a number of problems of varying character~\citep{KarvonenOates2018,Briol2019,Rathinavel2019}.
The theoretical results that we report may help to explain such empirical results previously observed.

In this example we study asymptotic overconfidence and underconfidence on $\Omega = [0,1] \subset \R$ using the released once integrated Brownian motion kernel
\begin{equation} \label{eq:BM1-kernel}
  K(x,y) = 1 + xy + \frac{1}{3} \min\{x,y\}^3 + \frac{1}{2} \abs[0]{x-y} \min\{x,y\}^2
\end{equation}
whose RKHS is $W_2^2([0,1])$ (in parlance of Section~\ref{sec:sobolev-kernels}, $\alpha=2$).
The term $1+xy$ ``releases'' the standard integrated Brownian motion by removing the requirement that $f(0) = f'(0) = 0$.
See~\citet[Section~10]{VaartZanten2008} and \citet[Section~2.2.3]{Karvonen2019} for details about integrated Brownian motion kernels.
For simplicity we only consider Bayesian quadrature for the computation of unweighted Lebesgue integrals on $[0,1]$.
Our integrands are of the form~\eqref{eq:matern-test-function} with fixed $m=3$, $\ell = 0.7$, $(a_1,a_2,a_3) = (1,2,0.5)$, and $(z_1,z_2,z_3) = (0.125,0.5,0.75)$.
Six different smoothness parameters are used: $\eta = n/4$ for $n=1,\ldots,6$.
As described in Section~\ref{sec:examples-mle}, this implies that the integrands are elements of $S^\beta([0,1])$ with $\beta = 2\eta + 1/2 = (n+1)/2$ for $n=1,\ldots,6$.
For each $N \geq 1$ the point set $X_N$ consists of the $N$ first elements in the van der Corput sequence~\eqref{eq:van-der-corput}.

The results are depicted in Figure~\ref{fig:coverage}.
In the right hand panel we see that asymptotic overconfidence appears to occur when $f$ is less smooth than the RKHS ($\eta = 0.25$ and $\eta = 0.50$), though it is not clear with which rate this happens.
When $\eta = 0.75$, which corresponds to $f$ being on the ``boundary'' of the RKHS, credible sets appear to be either slowly asymptotically overconfident or asymptotically honest.
When $f \in W_2^2([0,1])$ ($\eta = 1.00$ and $\eta = 1.25$) the GP model appears to be asymptotically honest.
Note that in all cases the relevant theoretical result, Theorem~\ref{thm:coverage-sobolev}, only guarantees overconfidence, if it happens, cannot happen too fast in that $\abs[0]{I(f)-Q_{X_N}(f)} / R_\BC(f,X_N) = O( N^{1/2} (\log N)^{1/\beta} )$.

\begin{figure}[t!]
  \centering
  \includegraphics{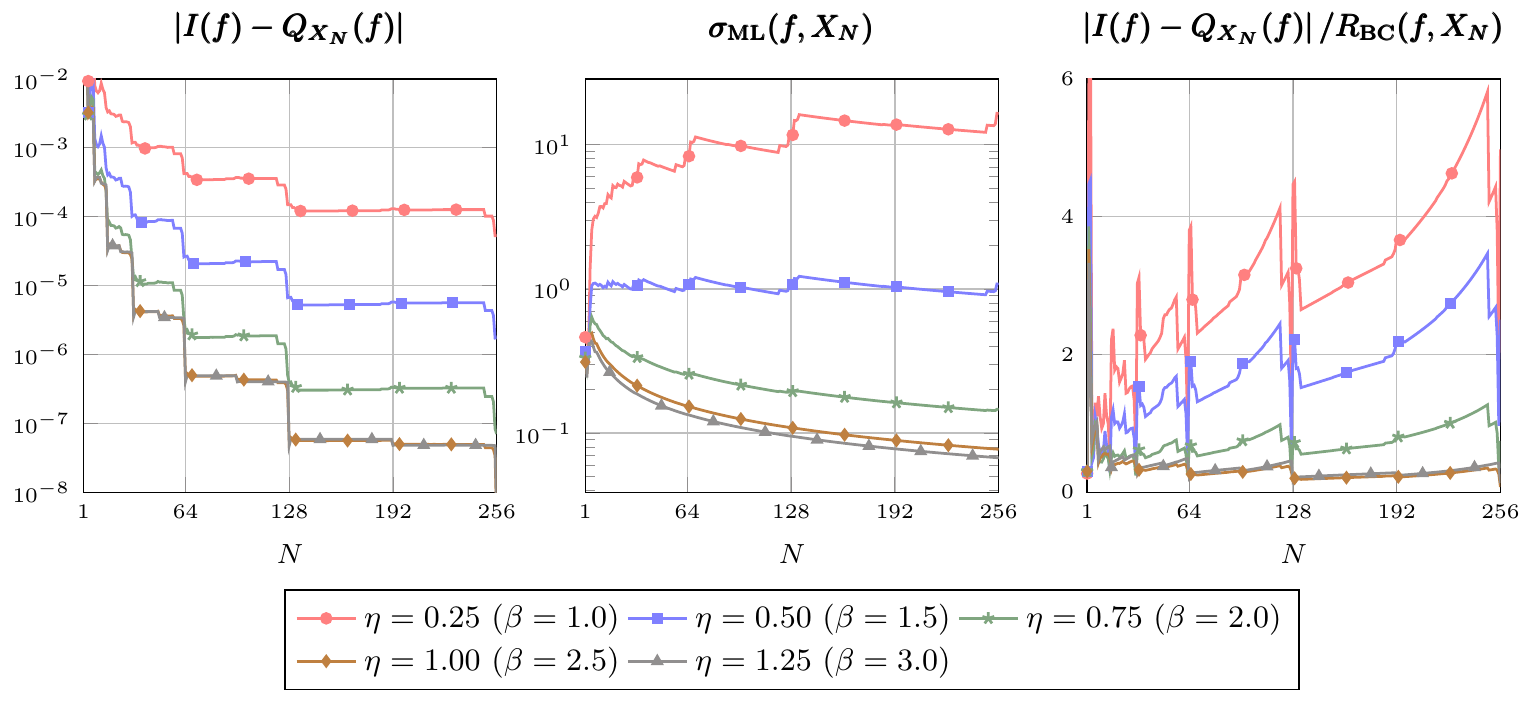}
  \caption{Behaviour of the integration error (\emph{left}), maximum likelihood estimate of the kernel scale parameter (\emph{centre}), and the standard score in~\eqref{eq:overconfident} and~\eqref{eq:underconfident} (\emph{right}) for a function $f$ of the form~\eqref{eq:matern-test-function} with $\eta = n/4$, $n=1,\ldots,6$, corresponding to $f \in S^\beta([0,1])$ for $\beta=(n+1)/2$, $n=1,\ldots,6$. The functions are evaluated at the nested van der Corput points~\eqref{eq:van-der-corput} for $N=1,\ldots,256$. The GP covariance kernel is the released integrated Brownian motion kernel~\eqref{eq:BM1-kernel} whose RKHS is $W_2^2([0,1])$.} \label{fig:coverage}
\end{figure}


\section{Conclusion} \label{sec: discussion}

In this article we analysed the asymptotic behaviour, as the number of data points grows, of maximum likelihood estimates of the scale parameter of a GP in the context of approximation and integration of a function that is exactly observed.
The results on maximum likelihood estimation were then used to show that in some settings the GP model can become at worst ``slowly'' overconfident.

Similar analysis of other common kernel parameters, in particular the length-scale parameter of a stationary kernel, and their effect on uncertainty quantification would be a logical next step.
Some work exists in the setting of the white noise model~\citep{Szabo2015,HadjiSzabo2019}.
However, such analysis is greatly complicated by the lack of a closed-form expression like~\eqref{eq:MLE-definition} for more general maximum likelihood estimators.
For the length-scale parameter there is some evidence that its MLE can converge to a constant $\ell_\infty \in (0,\infty)$ even when the GP model is misspecified~\citep{Karvonen-MLSP2019}.
Although proper selection of this parameter is often a prerequisite for an accurate and meaningful uncertainty quantification when $N$ is small, it would follow that the parameter has no effect on asymptotic overconfidence or underconfidence of the model.
\cite{Teckentrup2019} has recently proved approximation error bounds for GPs over compact sets of kernel parameters by assuming that the associated RKHS norm-equivalence constants can be uniformly bounded.
If the existence of a limit $\ell_\infty$ could be established, it is likely that results in \cite{Teckentrup2019} could be leveraged to extend the results in Section~\ref{sec:coverage-sobolev} to simultaneous maximum likelihood estimation of the scale and length-scale parameters.

Ther are also other popular approaches to kernel parameter selection. In \emph{marginalisation}, or \emph{full Bayes}, the scale parameter is treated as random and assigned an improper prior with density $p(\sigma^2) \propto 1/\sigma^2$ before being marginalised out~\citep[see][]{MacKay1996}.
If $N \geq 3$, the conditional process becomes a Student's $t$ process with~$N$ degrees of freedom, whose mean function is still $s_{f,X}$ but whose covariance function is now $(f_X^\T K_X^{-1} f_X/(N-2))  P_X(x,y)$.
The Student's $t$ distribution converges to a Gaussian when its degrees of freedom increases, which implies that the resulting posterior is indistinguishable from the one obtained using maximum likelihood in the large $N$ limit.
As a consequence, the asymptotic results of this article apply equally to the case where $\sigma$ is marginalised.
\emph{Cross-validation} offers more possibilities, both rooted~\citep{Fong2019} and not rooted~\citep[e.g.,][Section 2.2.3]{Rathinavel2019} in the GP model, to some of which our results on maximum likelihood may be relevant.
An empirical investigation on cross-validation has been performed in \citet{Bachoc2013}.

\section*{Acknowledgements}

We thank Gabriele Santin for pointing out some useful references.

\appendix

\section{Proofs for Sections~\ref{sec:Sobolev-rates} and~\ref{sec:MLE-Sobolev}} \label{sec:proofs}

This appendix contains proofs for the results in Sections~\ref{sec:Sobolev-rates} and~\ref{sec:MLE-Sobolev}.
Unlike in the statements of the results, here various constants are tracked carefully because these constants need to be controlled in the proofs of Theorem~\ref{thm:sobolev-rates-S} and Propositions~\ref{thm:MLE-upper-S} and~\ref{thm:MLE-lower-S}.


\begin{lemma} \label{lemma:bernstein} Suppose that $\alpha \geq \beta > d/2$ and $\Omega \subset \R^d$ satisfies Assumption~\ref{ass:Omega}. If $f \in W_2^\beta(\Omega)$ and $X \subset \Omega$ is a finite set of points, then there is $f_{\beta} \in W_2^\alpha(\Omega)$ such that $f_{\beta}|_X = f|_X$,
  \begin{equation} \label{eq:bernstein-statement}
    \norm[0]{f - f_{\beta}}_{W_2^\beta(\Omega)} \leq 5 \norm[0]{f}_{W_2^\beta(\Omega)}, \quad \text{ and } \quad \norm[0]{ f_{\beta} }_{W_2^\alpha(\Omega)} \leq C_\beta q_X^{\beta-\alpha} \norm[0]{ f }_{W_2^\beta(\Omega)},
  \end{equation}
  where the constant $C_\beta > 0$ does not depend on $f$ or $X$ and varies continuously with $\beta$.
\end{lemma}
\begin{proof}
For any band-limited $f_\sigma$ with band-limit $\sigma \geq 1$ we have
  \begin{equation} \label{eq:bernstein-intermediate}
    \begin{split} 
      \norm[0]{f_\sigma}_{W_2^\alpha(\R^d)}^2 &= \int_{\norm[0]{\xi} \leq \sigma} \big(1 + \norm[0]{\xi}^2\big)^\alpha \abs[0]{\widehat{f}_\sigma(\xi)}^2 \dif \xi \\
      &\leq (1+\sigma^2)^{\alpha-\beta} \int_{\norm[0]{\xi} \leq \sigma} \big(1 + \norm[0]{\xi}^2 \big)^\beta \abs[0]{\widehat{f}_\sigma(\xi)}^2 \dif \xi \\
      &= (1+\sigma^2)^{\alpha-\beta} \norm[0]{f_\sigma}_{W_2^\beta(\R^d)}^2 \\
      &\leq 2^{\alpha-\beta} \sigma^{2(\alpha-\beta)} \norm[0]{f_\sigma}_{W_2^\beta(\R^d)}^2.
      \end{split}
  \end{equation}
  Let $f_0 \in W_2^\beta(\R^d)$ be any extension of $f$. By Theorem~3.4 in \cite{Narcowich2006} there exists $f_{\beta} \in W_2^\alpha(\R^d)$ with band-width $\sigma = \kappa_{\beta} \, q_X^{-1}$ such that $f_{\beta}|_X = f|_X$ and
  \begin{equation} \label{eq:narcowich-5}
    \norm[0]{f_0 - f_{\beta}}_{W_2^\beta(\R^d)} \leq 5 \norm[0]{f_0}_{W_2^\beta(\R^d)}.
  \end{equation}
  The constant $\kappa_\beta > 0$ depends only on $d$ and $\beta$ and can be selected such that $\sigma=\kappa_\beta q_X^{-1} \geq 1$ for any points $X \subset \Omega$.
  That it varies continuously with~$\beta$ is ascertained by observing that, according to the proof of Lemma~3.3 in \citet{Narcowich2006}, it is a continuous combination of the constants $C_{\beta,d} = \Phi_\beta(0) + \sum_{n=1}^\infty 3d (n+2)^{d-1} \Phi_\beta(n)$, where $\Phi_\beta(x) = (1+x^2)^{-\beta}$ for $x \in \R$, and $c_{\beta,d}$, given in \citet[Theorem~12.3]{Wendland2005}, which are continuous functions of $\beta$.
  The second claim follows from~\eqref{eq:bernstein-intermediate} and~\eqref{eq:narcowich-5}:
  \begin{equation*}
    \begin{split}
      \norm[0]{f_{\beta}}_{W_2^\alpha(\R^d)} \leq 2^{(\alpha-\beta)/2} \sigma^{\alpha-\beta} \norm[0]{f_{\beta}}_{W_2^\beta(\R^d)} &\leq 2^{(\alpha-\beta)/2} \sigma^{\alpha-\beta} \big( \norm[0]{f_0}_{W_2^\beta(\R^d)} + \norm[0]{f_0 - f_{\beta}}_{W_2^\beta(\R^d)} \big)\\
      &\leq 6 \times 2^{(\alpha-\beta)/2} \kappa_\beta^{\alpha-\beta} q_X^{\beta-\alpha} \norm[0]{ f_0 }_{W_2^\beta(\R^d)}.
      \end{split}
  \end{equation*}
  That is, $C_\beta = 6 \times 2^{(\alpha-\beta)/2} \kappa_\beta^{\alpha-\beta}$.
  The Sobolev norms over $\R^d$ in the inequalities can be replaced with norms over $\Omega$ because the inequalities are valid for any extension of $f$.
\end{proof}

\begin{theorem}[{\citealt[Theorem~2.6]{WendlandRieger2005}}] \label{thm:wendland-rieger} 
  Let $\alpha \geq \beta$, $\floor{\beta} > d/2$, and $p \in [1,\infty]$. Suppose that $\Omega \subset \R^d$ satisfies Assumption~\ref{ass:Omega} and $K$ is a Sobolev kernel of order $\alpha$. If $f \in W_2^\beta(\Omega)$, then there are constants $C_\beta, h_{0,\beta} > 0$ such that
  \begin{equation*}
    \norm[0]{f - s_{f,X}}_{L^p(\Omega)} \leq C_\beta h_X^{\beta-d(1/2-1/p)_+} \norm[0]{f-s_{f,X}}_{W_2^\beta(\Omega)}
  \end{equation*}
  whenever $h_X \leq h_{0,\beta}$, where $(x)_+ \coloneqq \max\{x,0\}$.
  The constant $C_\beta$ depends on $d$, $\Omega$, $p$, and $\floor{\beta}$ and $h_{0,\beta}$ on $d$, $\Omega$, and $\floor{\beta}$.
\end{theorem}
\begin{proof}
The result as stated here follows by setting $k=\floor{\beta}$, $s=\beta-\floor{\beta}$, $p=2$, $q=p$, $m=0$, and $u=f-s_{f,X} \in W_2^\beta(\Omega)$ in Theorem~2.6 of \citet{WendlandRieger2005}.
\end{proof}

\begin{proof}[Proof of Theorem~\ref{thm:sobolev-rates}]

Theorem~\ref{thm:wendland-rieger} with $\beta = \alpha$ and $p=2$, the norm-equivalence~\eqref{eq:norm-equivalence}, and $\norm[0]{g-s_{g,X}}_{\mathcal{H}_K(\Omega)} \leq \norm[0]{g}_{\mathcal{H}_K(\Omega)}$ for any $g \in \mathcal{H}_K(\Omega)$ give $\norm[0]{g - s_{g,X}}_{W_2^\alpha(\Omega)} \leq C_K^{-1} C_K' \norm[0]{g}_{W_2^\alpha(\Omega)}$ and, if $h_X \leq h_{0,\alpha}$,
  \begin{equation*}
    \norm[0]{g - s_{g,X}}_{L^2(\Omega)} \leq C_\alpha h_X^\alpha \norm[0]{g - s_{g,X}}_{W_2^\alpha(\Omega)} \leq C_K^{-1} C_K' C_\alpha h_X^\alpha \norm[0]{g}_{W_2^\alpha(\Omega)}.
  \end{equation*}
Lemma~2.1 in \citet{Narcowich2006} therefore holds with the mapping $Tg = g-s_{g,X}$ and constants $\tau = \alpha$, $C_1 = C_K^{-1} C_K' C_\alpha h_X^\alpha$, and $C_2 = C_K^{-1} C_K'$.
It follows that
\begin{equation*}
  \norm[0]{T g}_{W_2^\beta(\Omega)} = \norm[0]{g - s_{g,X}}_{W_2^\beta(\Omega)} \leq C_1^{1-\beta/\alpha} C_2^{\beta/\alpha} \norm[0]{g}_{W_2^\alpha(\Omega)} = C_K^{-1} C_K' C_\alpha^{1-\beta/\alpha} h_X^{\alpha-\beta} \norm[0]{g}_{W_2^\alpha(\Omega)}
\end{equation*}
for $g \in W_2^\beta(\Omega)$ and $h_X \leq h_{0,\alpha}$.
Select now $g$ as the function $f_{\beta} \in W_2^\alpha(\Omega)$ in Lemma~\ref{lemma:bernstein} and let $C_\beta'$ be the constant in~\eqref{eq:bernstein-statement}.
Then, exploiting the fact that $f_{\beta}|_X = f|_X$ and thus $s_{f_{\beta},X} = s_{f,X}$,
  \begin{equation*}
    \begin{split}
      \norm[0]{f-s_{f,X}}_{W_2^\beta(\Omega)} &\leq \norm[0]{f-f_{\beta}}_{W_2^\beta(\Omega)} + \norm[0]{f_{\beta} - s_{f_{\beta},X}}_{W_2^\beta(\Omega)} + \norm[0]{s_{f_{\beta},X} - s_{f,X}}_{W_2^\beta(\Omega)} \\
      &\leq 5 \norm[0]{f}_{W_2^\beta(\Omega)} + C_K^{-1} C_K' C_\alpha^{1-\beta/\alpha} h_X^{\alpha-\beta} \norm[0]{f_{\beta}}_{W_2^\alpha(\Omega)} \\
      &\leq \big(5 + C_K^{-1} C_K' C_\alpha^{1-\beta/\alpha} C_\beta' \rho_X^{\alpha-\beta} \big) \norm[0]{f}_{W_2^\beta(\Omega)}.
    \end{split}
  \end{equation*}
Since the mesh ratio satisfies $\rho_X \geq 1$, we can write this as
  $\norm[0]{f-s_{f,X}}_{W_2^\beta(\Omega)} \leq C_\beta^* \rho_X^{\alpha-\beta} \norm[0]{f}_{W_2^\beta(\Omega)}$
for a constant $C_\beta^* > 0$ varying continuously with $\beta$.
Finally, Theorem~\ref{thm:wendland-rieger} yields
  \begin{equation} \label{eq:sobolev-rates-with-constant}
    \begin{split}
      \norm[0]{f-s_{f,X}}_{L^p(\Omega)} &\leq C_\beta h_X^{\beta-d(1/2-1/p)_+} \norm[0]{f-s_{f,X}}_{W_2^\beta(\Omega)} \\
      &\leq C_\beta C_\beta^* h_X^{\beta-d(1/2-1/p)_+} \rho_X^{\alpha-\beta} \norm[0]{f}_{W_2^\beta(\Omega)} \\
      \end{split}
  \end{equation}
  if $h_X \leq \min\{h_{0,\alpha}, h_{0,\beta}\}$.
  The claims of Theorem~\ref{thm:sobolev-rates} follow from the above inequality with $p=1$ and $p=\infty$, the inequality $\abs[0]{I(f) - Q_X(f)} \leq \sup_{x \in \Omega} w(x) \norm[0]{f-s_{f,X}}_{L^1(\Omega)}$ ($w$ is the weight function from Section~\ref{subsec: BC}), and that $\rho_{X_N}$ is bounded for quasi-uniform $(X_N)_{N=1}^\infty \subset \Omega$.
\end{proof}

Note that for any bounded set $B \subset [\ceil{d/2}, \infty)$ the constants related to~\eqref{eq:sobolev-rates-with-constant} satisfy
\begin{equation} \label{eq:well-behaved-constants}
  \sup_{ \beta \in B} C_\beta C_\beta^* < \infty \quad \text{ and } \quad \inf_{ \beta \in B} \min\{h_{0,\alpha},h_{0,\beta}\} > 0
\end{equation}
because, as remarked in the proof, $C_\beta^*$ is a continuous function of $\beta$ and $C_\beta$ and $h_{0,\beta}$, which are the constants in Theorem~\ref{thm:wendland-rieger}, depend on $\floor{\beta}$ and can thus take only a finite number of distinct values for $\beta \in B$.
This observation is used in the next proof.

\begin{proof}[Proof of Theorem~\ref{thm:sobolev-rates-S}] 
  The proof is based on the fact that $S_-^\beta(\R^d) \subset W_2^\gamma(\R^d)$ for every $\beta > \gamma$ and given here only for approximation. Let $f_0 \in S_-^\beta(\R^d) \cap W_2^\gamma(\R^d)$ be an extension of $f \in S_-^\beta(\Omega)$. For a quasi-uniform sequence $(X_N)_{N=1}^\infty \subset \Omega$ it follows from~\eqref{eq:sobolev-rates-with-constant} and~\eqref{eq:well-behaved-constants} that
  \begin{equation} \label{eq:sobolev-error-proof}
    \sup_{ x \in \Omega} \, \abs[0]{f(x) - s_{f,X_N}(x)} \leq C N^{-\gamma/d + 1/2} \norm[0]{f}_{W_2^\gamma(\Omega)} \leq C N^{-\gamma/d + 1/2} \norm[0]{f_0}_{W_2^\gamma(\R^d)}
  \end{equation}
  for every $\gamma \in B \coloneqq [\ceil{d/2}, \beta)$ and all $N \geq N_{0}$, where
  \begin{equation*}
    C \coloneqq \sup_{\gamma \in B} C_\gamma C_\gamma^* < \infty \quad \text{ and } \quad N_0 \coloneqq \bigg( C_\textsm{qu} \inf_{\gamma \in B} \min\{h_{0,\alpha}, h_{0,\gamma}\} \bigg)^{-d}
  \end{equation*}
are independent of $\gamma$ and $C_\textsm{qu} > 0$ is a constant such that $C_\textsm{qu}^{-1} N^{-1/d} \leq h_{X_N} \leq C_\textsm{qu} N^{-1/d}$ for all $N \geq 1$ (the existence of which follows from quasi-uniformity).
Set $\gamma = \gamma_N \coloneqq \beta - 1/\log N \to \beta$.
Because $f_0 \in S_-^\beta(\R^d)$, a spherical coordinate transform gives, with constants $C_1, C_2 > 0$ that depend on $\gamma$ and $\beta$, $d$, and $f_0$ and remain bounded away from zero and infinity as $\gamma \to \beta$,
\begin{equation*}
  \begin{split}
    \norm[0]{f_0}_{W_2^{\gamma}(\R^d)}^2 = \int_{\R^d} \big( 1 + \norm[0]{\xi}^2 \big)^\gamma \abs[0]{\widehat{f}_0(\xi)}^2 \dif \xi \leq C_1 \int_{\norm[0]{\xi} \geq 1} \norm[0]{\xi}^{2(\gamma-\beta) - d} \dif \xi &= C_1 C_2 \int_1^\infty r^{2(\gamma-\beta)-1} \dif r \\
    &= \frac{C_1 C_2}{2(\beta-\gamma)}.
  \end{split}
\end{equation*}
By inserting $\gamma = \gamma_N = \beta - 1/ \log N$ into~\eqref{eq:sobolev-error-proof} and exploiting the estimate above we thus get
\begin{equation*}
  \begin{split}
    \sup_{ x \in \Omega} \, \abs[0]{f(x) - s_{f,X_N}(x)} \lesssim N^{-\gamma_N/d + 1/2} \norm[0]{f}_{W_2^{\gamma_N}(\R^d)} &\lesssim N^{-\beta/d+1/2} N^{1/(d \log N)} (\log N)^{1/2} \\
    &= \e^{1/d} N^{-\beta/d - 1/2} (\log N)^{1/2},
  \end{split}
\end{equation*}
as claimed.
\end{proof}

\begin{proof}[Proof of Theorem~\ref{thm:WCE-exact}]

  The rates $\sup_{ x \in \Omega} P_{X_N}(x,x)^{1/2} \asymp N^{-\alpha/d + 1/2}$ and $V_{X_N}^{1/2} \asymp N^{-\alpha/d}$ follow the worst-case interpretation~\eqref{eq:WCE} of the standard deviations, Theorem~\ref{thm:sobolev-rates} with $\beta = \alpha$, and standard results on fundamental lower bounds for the rate of convergence of approximation and integration algorithms in Sobolev spaces, which can be found in \citet[Sections~1.3.11 and~1.3.12]{Novak1988}; \citet[Section~1.2, Chapter~VI]{Ritter2000}; and \citet[Section~4.2.4]{NovakWozniakowski2008}. We are left to prove the lower bound $P_{X_N}(x,x)^{1/2} \gtrsim N^{-\alpha/d + 1/2}$ for fixed $x$. Although this lower bound is more or less standard \citep[e.g.,][]{Schaback1995}, we have not found the exact version given here in the literature.
  
  By~\eqref{eq:WCE} the conditional standard deviation has the worst-case interpretation
  \begin{equation*}
    P_{X_N}(x,x)^{1/2} = \sup_{ \norm[0]{g}_{\mathcal{H}_K(\Omega)} \leq 1} \abs[0]{ g(x) - s_{g,X_N}(x)}.
  \end{equation*}
  If there is $g \in \mathcal{H}_K(\Omega)$ such that $g|_{X_N} \equiv 0$ it follows that $P_{X_N}(x,x)^{1/2} \geq \abs[0]{g(x)} \norm[0]{g}_{\mathcal{H}_K(\Omega)}^{-1}$ because in this case $s_{g,X_N} \equiv 0$.
  We follow the proof of Theorem~1 in \citet{DeMarchiSchaback2010}, a standard bump function argument, to construct this function.
  Let $\phi \colon \R^d \to \R$ be an infinitely smooth bump function that is supported on the unit ball and satisfies $\sup_{x \in \R^d} \phi(x) = \phi(0) = 1$.
  Let $\delta_{x,X_N} \coloneqq \min_{i=1,\ldots,N} \norm[0]{x-x_i}$ be the distance between $x \in \Omega$ and $X_N \subset \Omega$.
  Define $\phi_x \coloneqq \phi(\cdot - x)$ and $g_x \coloneqq \phi_x(\cdot / \delta_{x,X_N} ) \in W_2^\alpha(\Omega)$, which satisfies $g_x|_{X_N} \equiv 0$.
  Using the properties of the Fourier transform, a change of variables, and the fact that $\delta_{x,X_N} \leq h_{X_N} \leq 1$ for all sufficiently large $N$ due to quasi-uniformity we get
  \begin{equation*}
    \begin{split}
      \norm[0]{g_x}_{W_2^\alpha(\Omega)}^2 \leq \delta_{x,X_N}^{2d} \int_{\R^d} \big(1 + \norm[0]{\xi}^2 \big)^{\alpha} \abs[0]{\widehat{\phi}_x( \delta_{x,X_N} \xi)}^2 \dif \xi &= \delta_{x,X_N}^{d} \int_{\R^d} \bigg(1 + \frac{\norm[0]{\xi}^2}{\delta_{x,X_N}^2} \bigg)^{\alpha} \abs[0]{\widehat{\phi}_x(\xi)}^2 \dif \xi \\
      &\leq \delta_{x,X_N}^{d - 2\alpha} \int_{\R^d} \big(1+\norm[0]{\xi}^2 \big)^{\alpha} \abs[0]{\widehat{\phi}(\xi)}^2 \dif \xi \\
      &= \delta_{x,X_N}^{d-2\alpha} \norm[0]{\phi}_{W_2^\alpha(\R^d)}^2.
      \end{split}
  \end{equation*}
  By norm-equivalence we thus have $\norm[0]{g_x}_{\mathcal{H}_K(\Omega)} \leq C \delta_{x,X_N}^{-\alpha+d/2}$ for a constant $C > 0$ which is independent of $x$ and $X_N$.
  Therefore
  \begin{equation} \label{eq:PX-lower-in-proof}
    P_{X_N}(x,x)^{1/2} \geq C^{-1} \delta_{x,X_N}^{\alpha-d/2}.    
  \end{equation}
  
  Because the point set sequence is quasi-uniform, there is a constant $C_\textsm{qu} > 0$ such that $C_\textsm{qu}^{-1} N^{-1/d} \leq q_{X_N}$ for all $N \geq 1$.
  If $\delta_{x,X_N} \geq q_{X_N}$, then $\delta_{x,X_N} \geq c_1 N^{-1/d}$ and~\eqref{eq:PX-lower-in-proof} gives
  \begin{equation} \label{eq:PXN-lower-in-proof}
    P_{X_N}(x,x)^{1/2} \geq C^{-1} \delta_{x,X_N}^{\alpha-d/2} \geq C^{-1} c_1^{-\alpha+d/2} N^{-\alpha/d+1/2},
  \end{equation}
  the claimed lower bound.
  Assume thus that $\delta_{x,X_N} \leq q_{X_N}$.
  Let $x^* \in X_N$ be a point closest to $x \notin X_N$. Then $2 q_{X_N} \leq \norm[0]{x^* - x} + \norm[0]{x - x'} = \delta_{x,X_N} + \norm[0]{x-x'}$ for any $x' \in X_N \setminus \{x^*\}$.
  If~$N$ is such that $\delta_{x,X_N} < \delta_{x,X_{N-1}}$ there is $x'$ such that $\norm[0]{x-x'} = \delta_{x,X_{N-1}}$.
  For such~$N$ we thus have $2q_{X_N} \leq \delta_{x,X_N} + \delta_{x,X_{N-1}}$, which together with $\delta_{x,X_N} \geq q_{X_N}$ and quasi-uniformity yields
  \begin{equation*}
    \delta_{x,X_{N-1}} \geq C_\textsm{qu}^{-1} N^{-1/d} \geq 2^{-1/d} C_\textsm{qu}^{-1} (N-1)^{-1/d}
  \end{equation*}
  for $N \geq 2$. Again, a lower bound of the form~\eqref{eq:PXN-lower-in-proof} thus holds.
  This completes the proof.
\end{proof}


\begin{proof}[Proof of Proposition~\ref{thm:MLE-upper}]
  For any distinct points $X \subset \Omega$ Lemma~\ref{lemma:bernstein}, $s_{f,X} = s_{f_{\beta},X}$, the minimum-norm property of the GP conditional mean, and the norm-equivalence~\eqref{eq:norm-equivalence} yield
  \begin{equation*}
    \norm[0]{s_{f,X}}_{\mathcal{H}_K(\Omega)} = \norm[0]{s_{f_{\beta},X}}_{\mathcal{H}_K(\Omega)} \leq \norm[0]{f_{\beta}}_{\mathcal{H}_K(\Omega)} \leq C_K' \norm[0]{f_{\beta}}_{W_2^\alpha(\Omega)} \leq C_K' C_\beta q_{X}^{\beta-\alpha} \norm[0]{f}_{W_2^\beta(\Omega)}.
  \end{equation*}
  The claims follow from~\eqref{eq:MLE} and the fact that $q_{X_N} \gtrsim N^{-1/d}$ for quasi-uniform points.
\end{proof}

\begin{proof}[Proof of Proposition~\ref{thm:MLE-upper-S}]
The proof is similar to that of Theorem~\ref{thm:sobolev-rates-S}. The use of Theorem~\ref{thm:sobolev-rates} is replaced with Proposition~\ref{thm:MLE-upper} which implies that
\begin{equation*}
  \sigma_\ML(f,X_N) \leq C_{\gamma} \, N^{(\alpha-\gamma)/d-1/2} \norm[0]{f_0}_{W_2^\gamma(\R^d)},
\end{equation*}
where $f_0 \in S_-^\beta(\R^d) \cap W_2^\gamma(\R^d)$ is an extension of $f$ and $C_\gamma$ again satisfies $\sup_{\gamma \in [\ceil{d/2},\beta)} C_\gamma < \infty$.
\end{proof}


\begin{proof}[Proof of Proposition~\ref{thm:MLE-lower}] 
This proof is adapted from the proof of Theorem~8 in \cite{VaartZanten2011}.
Let $X \subset \Omega$ be any distinct points. By Theorem~\ref{thm:sobolev-rates} there are $C_\gamma, h_{0,\gamma} > 0$ such that for $h_{X} \leq h_{0,\gamma}$,
  \begin{equation} \label{eq:epsilonX}
    \norm[0]{f-s_{f,X}}_{L^{2}(\Omega)} \leq C_{\gamma} h_{X}^{\gamma}\rho_{X}^{\alpha-\gamma}\norm[0]{f_0}_{W^{\gamma}_{2}(\R^d)} \eqqcolon \varepsilon_X.
  \end{equation}
  In the proof of Theorem~\ref{thm:sobolev-rates-S} it is shown that $C_\gamma$ and $h_{0,\gamma}$ are bounded away from zero and infinity if $\gamma$ remains in a bounded interval.
  Because the support of $f_0$ is compact and contained in the interior of $\Omega$, there is a non-negative bump function $\phi \colon \R^d \to \R$ such that $\phi|_{\supp(f_0)} \equiv 1$, $\phi|_{\R^d \setminus \inte(\Omega)} \equiv 0$, $\sup_{ x \in \Omega} \phi(x) = 1$, and $\abs[0]{\widehat{\phi}(\b{\xi})\e^{\norm[0]{\b{\xi}}^{u}}} \to 0$ as $\norm[0]{\b{\xi}} \to \infty$ for some $u > 0$.
  Let $s_{f,X,0} \in W_2^\alpha(\R^d)$ be an extension of $s_{f,X} \in W_2^\alpha(\Omega)$.
  By Parseval's identity and $f_0 = f_0 \phi$,
\begin{equation} \label{eq:epsilon-convolution}
  \norm[0]{\widehat{f_0} - \widehat{s}_{f,X,0}*\widehat{\phi}}_{L^{2}(\mathbb{R}^{d})} = \norm[0]{f_0 - s_{f,X,0}\phi}_{L^{2}(\mathbb{R}^{d})} \leq \norm[0]{f - s_{f,X}}_{L^{2}(\Omega)} \leq \varepsilon_X.
\end{equation}
For $R > 0$ let $\mathbbm{1}_{R}^\textsf{c}$ be the indicator function of the set $\Set{\b{x}\in\R^{d}}{ \norm[0]{\b{x}} > R}$. Because $f_0 \in S_+^\beta(\R^d)$, for sufficiently large $R$ we have
\begin{equation*}
    \norm[0]{\widehat{f_0} \mathbbm{1}_{2R}^\textsf{c}}_{L^{2}(\mathbb{R}^{d})}^2 = \int_{\norm[0]{\xi} > 2R} \abs[0]{\widehat{f_0}(\xi)}^2 \dif \xi \geq C_f \int_{\norm[0]{\xi} > 2R} \norm[0]{\xi}^{-2\beta-d} \dif \xi \geq C_f \widetilde{C} R^{-2\beta} \eqqcolon C_1 R^{-2\beta},
\end{equation*}
where $C_f > 0$ depends on $f_0$ and $\widetilde{C} > 0$ on $\beta$ and $d$.
The reverse triangle inequality and~\eqref{eq:epsilon-convolution} thus yield
\begin{equation*}
  \norm[0]{(\widehat{s}_{f,X,0}*\widehat{\phi})\mathbbm{1}_{2R}^\textsf{c}}_{L^{2}(\mathbb{R}^{d})} \geq \norm[0]{\widehat{f_0} \mathbbm{1}_{2R}^\textsf{c}}_{L^{2}(\mathbb{R}^{d})} - \norm[0]{( \widehat{f_0} - \widehat{s}_{f,X,0}*\widehat{\phi}) \mathbbm{1}_{2R}^\textsf{c}}_{L^{2}(\mathbb{R}^{d})} \geq C_{1} R^{-\beta}-\varepsilon_X.
\end{equation*}
By Lemma~16 in \cite{VaartZanten2011},
\begin{equation} \label{eq:lemma16}
  \norm[0]{\widehat{s}_{f,X,0}\mathbbm{1}_{R}^\textsf{c}}_{L^{2}(\mathbb{R}^{d})} \norm[0]{\widehat{\phi}(1-\mathbbm{1}_{R}^\textsf{c})}_{L^{1}(\mathbb{R}^{d})}\geq C_{1} R^{-\beta} -\varepsilon_X - \norm[0]{\widehat{s}_{f,X,0}}_{L^{2}(\mathbb{R}^{d})} \norm[0]{\widehat{\phi} \mathbbm{1}_{R}^\textsf{c}}_{L^{1}(\mathbb{R}^{d})}.
\end{equation}
Let $m(\b{\xi}) \defeq (1+\norm[0]{\b{\xi}}^{2})^{\alpha/2}$ so that $\norm[0]{s_{f,X,0}}_{W^{\alpha}_{2}(\R^{d})} = \norm[0]{\widehat{s}_{f,X,0}m}_{L^{2}(\R^{d})}$. We engage in slight abuse of notation by also writing $m(r) = (1+r^{2})^{\alpha/2}$ for $r \in \R$. By the definition of Sobolev norm,
\begin{align*}
  \norm[0]{\widehat{s}_{f,X,0}\mathbbm{1}_{R}^\textsf{c}}_{L^{2}(\mathbb{R}^{d})} = \norm[0]{\widehat{s}_{f,X,0}m \mathbbm{1}_{R}^\textsf{c} m^{-1}}_{L^{2}(\mathbb{R}^{d})} \leq m(R)^{-1}\norm[0]{s_{f,X,0}}_{W^{\alpha}_{2}(\mathbb{R}^{d})}
\end{align*}
and $\norm[0]{\widehat{s}_{f,X,0}}_{L^{2}(\R^{d})} \leq \norm[0]{s_{f,X,0}}_{W^{\alpha}_{2}(\R^{d})}$. Set $R = C_1^{1/\beta} (2\varepsilon_X)^{-1/\beta}$ and use these estimates to rearrange~\eqref{eq:lemma16} as
\begin{equation*}
  \Big[ m(R)^{-1}\norm[0]{\widehat{\phi}(1-\mathbbm{1}_{R}^\textsf{c})}_{L^{1}(\mathbb{R}^{d})} + \norm[0]{\widehat{\phi}\mathbbm{1}_{R}^\textsf{c}}_{L^{1}(\mathbb{R}^{d})} \Big] \norm[0]{s_{f,X,0}}_{W^{\alpha}_{2}(\mathbb{R}^{d})} \geq C_{1}R^{-\beta} -\varepsilon_X = \varepsilon_X.
\end{equation*}
By construction, $\widehat{\phi}\in L^{1}(\mathbb{R}^{d})$ and $\norm[0]{\widehat{\phi}\mathbbm{1}_{R}^\textsf{c}}_{L^{1}(\mathbb{R}^{d})} \leq C_{2} \e^{-dR^{u}}$ for some constant $C_{2} > 0$. Let $C_{3} > 0$ be a constant such that $C_2 \e^{-dr^{u}} \leq C_3 m(r)^{-1}$ for all $r \geq 0$. Then
\begin{align*}
  \norm[0]{s_{f,X,0}}_{W^{\alpha}_{2}(\mathbb{R}^{d})} & \geq \varepsilon_X \Big[ m(R)^{-1}\norm[0]{\widehat{\phi}(1-\mathbbm{1}_{R}^\textsf{c})}_{L^{1}(\mathbb{R}^{d})} + \norm[0]{\widehat{\phi}\mathbbm{1}_{R}^\textsf{c}}_{L^{1}(\mathbb{R}^{d})} \Big]^{-1} \\
  & \geq \varepsilon_X \Big[ m(R)^{-1} \norm[0]{\widehat{\phi}}_{L^{1}(\mathbb{R}^{d})} + C_{2}e^{-dR^{u}} \Big]^{-1} \\
  & \geq \varepsilon_X \Big[ m(R)^{-1} \big( \norm[0]{\widehat{\phi}}_{L^{1}(\mathbb{R}^{d})}+C_{3} \big) \Big]^{-1} \\
  &\geq C_4 \varepsilon_X^{1-\alpha/\beta},
\end{align*}
where $C_4 = 2^{-1/\beta} C_1^{1/\beta} (\norm[0]{\widehat{\phi}}_{L^{1}(\mathbb{R}^{d})}+C_{3})^{-1}$ does not depend on $\gamma$.
The definition of $\varepsilon_X$ in~\eqref{eq:epsilonX} therefore gives
\begin{align*}
  \norm[0]{s_{f,X,0}}_{W^{\alpha}_{2}(\mathbb{R}^{d})} &\geq C_4 \big( C_\gamma h_{X}^{\gamma}\rho_{X}^{\alpha-\gamma}\norm[0]{f_0}_{W^{\gamma}_{2}(\R^d)} \big)^{1-\alpha/\beta} \\
  &= C_4 C_\gamma^{1-\alpha/\beta} h_X^{\gamma(1-\alpha/\beta)} \rho_X^{-(\alpha-\gamma)(\alpha-\beta)/\beta} \norm[0]{f_0}_{W^\gamma_2(\R^d)}^{1-\alpha/\beta}. 
\end{align*}
If $(X_{N})_{N=1}^\infty$ is a quasi-uniform sequence, $\rho_{X_N}$ remains bounded and $h_{X_N} \gtrsim N^{-1/d}$. Thus
\begin{align*}
  \norm[0]{s_{f,X_N,0}}_{W^{\alpha}_{2}(\mathbb{R}^{d})} \gtrsim N^{\gamma(\alpha/\beta-1)/d}\norm[0]{f_0}_{W^{\gamma}_{2}(\R^d)}^{1-\alpha/\beta}.
\end{align*}
The claims now follow from the norm-equivalence of $\mathcal{H}_K(\Omega)$ and $W_2^\alpha(\Omega)$, the Sobolev extension theorem~\citep[Theorem~1.4.3.1]{Grisvald1985}, and~\eqref{eq:MLE}.
\end{proof}

\begin{proof}[Proof of Proposition~\ref{thm:MLE-lower-S}]

Let $\gamma_N = \beta - 1/\log N$. The arguments in the proof of Theorem~\ref{thm:MLE-upper-S}, the Sobolev extension theorem, and~\eqref{eq:MLE-lower-uniform} yield
\begin{equation*}
  \begin{split}
  \sigma_\ML(f,X_N) &\gtrsim N^{\gamma_N(\alpha/\beta-1)/d - 1/2}\norm[0]{f_0}_{W^{\gamma_N}_{2}(\R^d)}^{1-\alpha/\beta} \\ 
  &\gtrsim N^{(\alpha-\beta)/d-1/2} N^{(1-\alpha/\beta)/(d \log N)} (\log N)^{(1-\alpha/\beta)/2} \\
  &= \e^{(1-\alpha/\beta)/d} N^{(\alpha-\beta)/d-1/2} (\log N)^{(1-\alpha/\beta)/2},
  \end{split}
\end{equation*}
which is the claimed bound.
\end{proof}



\begin{thebibliography}{}

\bibitem[\protect\astroncite{Arcang{\'e}li et~al.}{2007}]{Arcangeli2007}
Arcang{\'e}li, R., de~Silanes, M. C.~L., and Torrens, J.~J. (2007).
\newblock An extension of a bound for functions in {S}obolev spaces, with
  applications to $(m,s)$-spline interpolation and smoothing.
\newblock {\em Numerische Mathematik}, 107(2):181--211.

\bibitem[\protect\astroncite{Arcang{\'e}li et~al.}{2012}]{Arcangeli2012}
Arcang{\'e}li, R., de~Silanes, M. C.~L., and Torrens, J.~J. (2012).
\newblock Extension of sampling inequalities to {S}obolev semi-norms of
  fractional order and derivative data.
\newblock {\em Numerische Mathematik}, 121(3):587--608.

\bibitem[\protect\astroncite{Atkinson}{1989}]{Atkinson1989}
Atkinson, K.~E. (1989).
\newblock {\em An Introduction to Numerical Analysis}.
\newblock John Wiley \& Sons, 2nd edition.

\bibitem[\protect\astroncite{Bach}{2017}]{Bach2017}
Bach, F. (2017).
\newblock On the equivalence between kernel quadrature rules and random feature
  expansions.
\newblock {\em Journal of Machine Learning Research}, 18(19):1--38.

\bibitem[\protect\astroncite{Bachoc}{2013}]{Bachoc2013}
Bachoc, F. (2013).
\newblock Cross validation and maximum likelihood estimations of
  hyper-parameters of {G}aussian processes with model misspecification.
\newblock {\em Computational Statistics \& Data Analysis}, 66:55--69.

\bibitem[\protect\astroncite{Bachoc}{2017}]{Bachoc2018}
Bachoc, F. (2017).
\newblock Asymptotic analysis of covariance parameter estimation for {G}aussian
  processes in the misspecified case.
\newblock {\em Bernoulli}, 24(2):1531--1575.

\bibitem[\protect\astroncite{Bachoc et~al.}{2019}]{Bachoc2019}
Bachoc, F., Lagnoux, A., and Lopera-L\'opez, A. (2019).
\newblock Maximum likelihood estimation for {G}aussian processes under
  inequality constraints.
\newblock {\em Electronic Journal of Statistics}, 13(2):2921--2969.

\bibitem[\protect\astroncite{Bachoc et~al.}{2017}]{Bachoc2017}
Bachoc, F., Lagnoux, A., and Nguyen, T. M.~N. (2017).
\newblock Cross-validation estimation of covariance parameters under
  fixed-domain asymptotics.
\newblock {\em Journal of Multivariate Analysis}, 160:42--67.

\bibitem[\protect\astroncite{B{\u{a}}z{\u{a}}van et~al.}{2012}]{Bazavan2012}
B{\u{a}}z{\u{a}}van, E.~G., Li, F., and Sminchisescu, C. (2012).
\newblock Fourier kernel learning.
\newblock In {\em European Conference on Computer Vision}, pages 459--473.
  Springer.

\bibitem[\protect\astroncite{Berlinet and
  Thomas{-}Agnan}{2004}]{BerlinetThomasAgnan2004}
Berlinet, A. and Thomas{-}Agnan, C. (2004).
\newblock {\em Reproducing Kernel {H}ilbert Spaces in Probability and
  Statistics}.
\newblock Springer.

\bibitem[\protect\astroncite{Bogachev}{1998}]{Bogachev1998}
Bogachev, V.~I. (1998).
\newblock {\em Gaussian Measures}.
\newblock American Mathematical Society.

\bibitem[\protect\astroncite{Briol et~al.}{2019}]{Briol2019}
Briol, F.-X., Oates, C.~J., Girolami, M., Osborne, M.~A., and Sejdinovic, D.
  (2019).
\newblock Probabilistic integration: A role in statistical computation? (with
  discussion and rejoinder).
\newblock {\em Statistical Science}, 34(1):1--22.

\bibitem[\protect\astroncite{Bull}{2011}]{Bull2011}
Bull, A.~D. (2011).
\newblock Convergence rates of efficient global optimization algorithms.
\newblock {\em Journal of Machine Learning Research}, 12:2879--2904.

\bibitem[\protect\astroncite{Cockayne et~al.}{2019}]{Cockayne2019}
Cockayne, J., Oates, C., Sullivan, T., and Girolami, M. (2019).
\newblock {B}ayesian probabilistic numerical methods.
\newblock {\em SIAM Review}, 61(4):756--789.

\bibitem[\protect\astroncite{Dashti and Stuart}{2017}]{DashtiStuart2017}
Dashti, M. and Stuart, A.~M. (2017).
\newblock The {B}ayesian approach to inverse problems.
\newblock In {\em Handbook of Uncertainty Quantification}, pages 311--428.
  Springer.

\bibitem[\protect\astroncite{De~Marchi and
  Schaback}{2010}]{DeMarchiSchaback2010}
De~Marchi, S. and Schaback, R. (2010).
\newblock Stability of kernel-based interpolation.
\newblock {\em Advances in Computational Mathematics}, 32(2):155--161.

\bibitem[\protect\astroncite{Dong}{2012}]{Dong2012}
Dong, D. (2012).
\newblock Mine gas emission prediction based on {G}aussian process model.
\newblock {\em Procedia Engineering}, 45:334--338.

\bibitem[\protect\astroncite{Driscoll}{1973}]{Driscoll1973}
Driscoll, M.~F. (1973).
\newblock The reproducing kernel {H}ilbert space structure of the sample paths
  of a {G}aussian process.
\newblock {\em Zeitschrift f{\"u}r Wahrscheinlichkeitstheorie und Verwandte
  Gebiete}, 26(4):309--316.

\bibitem[\protect\astroncite{Duvenaud}{2014}]{Duvenaud2014}
Duvenaud, D. (2014).
\newblock {\em Automatic Model Construction with {G}aussian Processes}.
\newblock PhD thesis, University of Cambridge.

\bibitem[\protect\astroncite{Fasshauer and
  McCourt}{2015}]{FasshauerMcCourt2015}
Fasshauer, G. and McCourt, M. (2015).
\newblock {\em Kernel-based Approximation Methods Using {MATLAB}}.
\newblock World Scientific Publishing.

\bibitem[\protect\astroncite{Fasshauer}{2011}]{Fasshauer2011}
Fasshauer, G.~E. (2011).
\newblock Positive definite kernels: past, present and future.
\newblock {\em Dolomite Research Notes on Approximation}, 4:21--63.

\bibitem[\protect\astroncite{Fong and Holmes}{2019}]{Fong2019}
Fong, E. and Holmes, C.~C. (2019).
\newblock On the marginal likelihood and cross-validation.
\newblock {\em arXiv:1905.08737v2}.

\bibitem[\protect\astroncite{Gao et~al.}{2008}]{Gao2008}
Gao, P., Honkela, A., Rattray, M., and Lawrence, N. (2008).
\newblock Gaussian process modelling of latent chemical species: applications
  to inferring transcription factor activities.
\newblock {\em Bioinformatics}, 24(16):i70--i75.

\bibitem[\protect\astroncite{Grisvald}{1985}]{Grisvald1985}
Grisvald, P. (1985).
\newblock {\em Elliptic Problems in Nonsmooth Domains}.
\newblock Pitman Publishing.

\bibitem[\protect\astroncite{Hadji and Szab\'{o}}{2019}]{HadjiSzabo2019}
Hadji, A. and Szab\'{o}, B. (2019).
\newblock Can we trust {B}ayesian uncertainty quantification from {G}aussian
  process priors with squared exponential covariance kernel?
\newblock {\em arXiv:1904.01383v1}.

\bibitem[\protect\astroncite{Hennig et~al.}{2015}]{Hennig2015}
Hennig, P., Osborne, M.~A., and Girolami, M. (2015).
\newblock Probabilistic numerics and uncertainty in computations.
\newblock {\em Proceedings of the Royal Society of London A: Mathematical,
  Physical and Engineering Sciences}, 471(2179).

\bibitem[\protect\astroncite{Iske}{2018}]{Iske2018}
Iske, A. (2018).
\newblock {\em Approximation Theory and Algorithms for Data Analysis}.
\newblock Springer.

\bibitem[\protect\astroncite{Kanagawa et~al.}{2018}]{Kanagawa2018}
Kanagawa, M., Hennig, P., Sejdinovic, D., and Sriperumbudur, B.~K. (2018).
\newblock Gaussian processes and kernel methods: A review on connections and
  equivalences.
\newblock {\em arXiv:1807.02582v1}.

\bibitem[\protect\astroncite{Kanagawa et~al.}{2020}]{Kanagawa2019}
Kanagawa, M., Sriperumbudur, B.~K., and Fukumizu, K. (2020).
\newblock Convergence analysis of deterministic kernel-based quadrature rules
  in misspecified settings.
\newblock {\em Foundations of Computational Mathematics}, 20:155--194.

\bibitem[\protect\astroncite{Karvonen}{2019}]{Karvonen2019}
Karvonen, T. (2019).
\newblock {\em Kernel-Based and {B}ayesian Methods for Numerical Integration}.
\newblock PhD thesis, Department of Electrical Engineering and Automation,
  Aalto University.

\bibitem[\protect\astroncite{Karvonen et~al.}{2018}]{KarvonenOates2018}
Karvonen, T., Oates, C.~J., and Särkkä, S. (2018).
\newblock A {B}ayes--{S}ard cubature method.
\newblock In {\em Advancges in Neural Information Processing Systems 31}, pages
  5882--5893.

\bibitem[\protect\astroncite{Karvonen et~al.}{2019}]{Karvonen-MLSP2019}
Karvonen, T., Tronarp, F., and Särkkä, S. (2019).
\newblock Asymptotics of maximum likelihood parameter estimates for {G}aussian
  processes: {T}he {O}rnstein--{U}hlenbeck prior.
\newblock In {\em Proceedings of the 29th IEEE International Workshop on
  Machine Learning for Signal Processing}.

\bibitem[\protect\astroncite{Kennedy and O'Hagan}{2001}]{Kennedy2001}
Kennedy, M. and O'Hagan, A. (2001).
\newblock Bayesian calibration of computer models.
\newblock {\em Journal of the Royal Statistical Society: Series B (Statistical
  Methodology)}, 63(3):425--464.

\bibitem[\protect\astroncite{Kowalska and Peel}{2012}]{Kowalska2012}
Kowalska, K. and Peel, L. (2012).
\newblock Maritime anomaly detection using {G}aussian process active learning.
\newblock In {\em Proceedings of the 15th International Conference on
  Information Fusion}, pages 1164--1171.

\bibitem[\protect\astroncite{Larkin}{1972}]{Larkin1972}
Larkin, F.~M. (1972).
\newblock {G}aussian measure in {H}ilbert space and applications in numerical
  analysis.
\newblock {\em Rocky Mountain Journal of Mathematics}, 2(3):379--422.

\bibitem[\protect\astroncite{Liebl and Reimherr}{2019}]{Leibl2019}
Liebl, D. and Reimherr, M. (2019).
\newblock Fast and fair simultaneous confidence bands for functional
  parameters.
\newblock {\em arXiv:1910.00131v2}.

\bibitem[\protect\astroncite{Lindgren et~al.}{2011}]{Lindgren2011}
Lindgren, F., Rue, H., and Lindstr\"{o}m, J. (2011).
\newblock An explicit link between {G}aussian fields and {G}aussian {M}arkov
  random fields: {T}he stochastic partial differential equation approach.
\newblock {\em Journal of the Royal Statistical Society: Series B (Statistical
  Methodology)}, 73(4):423--498.

\bibitem[\protect\astroncite{Liu et~al.}{2013}]{Liu2013}
Liu, D., Pang, J., Zhou, J., Peng, Y., and Pecht, M. (2013).
\newblock Prognostics for state of health estimation of lithium-ion batteries
  based on combination {G}aussian process functional regression.
\newblock {\em Microelectronics Reliability}, 53(6):832--839.

\bibitem[\protect\astroncite{Luki{\'c} and Beder}{2001}]{LukicBeder2001}
Luki{\'c}, M.~N. and Beder, J.~H. (2001).
\newblock Stochastic processes with sample paths in reproducing kernel
  {H}ilbert spaces.
\newblock {\em Transactions of the American Mathematical Society},
  353(10):3945--3969.

\bibitem[\protect\astroncite{MacKay}{1992}]{MacKay1992}
MacKay, D.~J. (1992).
\newblock Bayesian interpolation.
\newblock {\em Neural Computation}, 4(3):415--447.

\bibitem[\protect\astroncite{MacKay}{1996}]{MacKay1996}
MacKay, D.~J. (1996).
\newblock Hyperparameters: {O}ptimize, or integrate out?
\newblock In {\em Maximum Entropy and Bayesian Methods}, pages 43--59.
  Springer.

\bibitem[\protect\astroncite{Manogaran and Lopez}{2018}]{Manogaran2018}
Manogaran, G. and Lopez, D. (2018).
\newblock A {G}aussian process based big data processing framework in cluster
  computing environment.
\newblock {\em Cluster Computing}, 21(1):189--204.

\bibitem[\protect\astroncite{Narcowich et~al.}{2006}]{Narcowich2006}
Narcowich, F.~J., Ward, J.~D., and Wendland, H. (2006).
\newblock {S}obolev error estimates and a {B}ernstein inequality for scattered
  data interpolation via radial basis functions.
\newblock {\em Constructive Approximation}, 24(2):175--186.

\bibitem[\protect\astroncite{Novak}{1988}]{Novak1988}
Novak, E. (1988).
\newblock {\em Deterministic and Stochastic Error Bounds in Numerical
  Analysis}.
\newblock Springer-Verlag.

\bibitem[\protect\astroncite{Novak and
  Wo\'{z}niakowski}{2008}]{NovakWozniakowski2008}
Novak, E. and Wo\'{z}niakowski, H. (2008).
\newblock {\em Tractability of Multivariate Problems. {V}olume {I}: {L}inear
  Information}.
\newblock European Mathematical Society.

\bibitem[\protect\astroncite{Oettershagen}{2017}]{Oettershagen2017}
Oettershagen, J. (2017).
\newblock {\em Construction of Optimal Cubature Algorithms with Applications to
  Econometrics and Uncertainty Quantification}.
\newblock PhD thesis, University of Bonn.

\bibitem[\protect\astroncite{O'Hagan}{1991}]{OHagan1991}
O'Hagan, A. (1991).
\newblock Bayes--{H}ermite quadrature.
\newblock {\em Journal of Statistical Planning and Inference}, 29(3):245--260.

\bibitem[\protect\astroncite{Oliva et~al.}{2016}]{Oliva2016}
Oliva, J.~B., Dubey, A., Wilson, A.~G., P{\'o}czos, B., Schneider, J., and
  Xing, E.~P. (2016).
\newblock Bayesian nonparametric kernel-learning.
\newblock In {\em Proceedings of the 19th International Conference on
  Artificial Intelligence and Statistics}, pages 1078--1086.

\bibitem[\protect\astroncite{Rajpaul et~al.}{2015}]{Rajpaul2015}
Rajpaul, V., Aigrain, S., Osborne, M., Reece, S., and Roberts, S. (2015).
\newblock A {G}aussian process framework for modelling stellar activity signals
  in radial velocity data.
\newblock {\em Monthly Notices of the Royal Astronomical Society},
  452(3):2269--2291.

\bibitem[\protect\astroncite{Rasmussen and
  Williams}{2006}]{RasmussenWilliams2006}
Rasmussen, C.~E. and Williams, C. K.~I. (2006).
\newblock {\em Gaussian Processes for Machine Learning}.
\newblock Adaptive Computation and Machine Learning. MIT Press.

\bibitem[\protect\astroncite{Rathinavel and Hickernell}{2019}]{Rathinavel2019}
Rathinavel, J. and Hickernell, F. (2019).
\newblock Fast automatic {B}ayesian cubature using lattice sampling.
\newblock {\em Statistics and Computing}, 29(6):1215--1229.

\bibitem[\protect\astroncite{Rieger and Zwicknagl}{2010}]{RiegerZwicknagl2010}
Rieger, C. and Zwicknagl, B. (2010).
\newblock Sampling inequalities for infinitely smooth functions, with
  applications to interpolation and machine learning.
\newblock {\em Advances in Computational Mathematics}, 32:103--129.

\bibitem[\protect\astroncite{Ritter}{2000}]{Ritter2000}
Ritter, K. (2000).
\newblock {\em Average-Case Analysis of Numerical Problems}.
\newblock Springer.

\bibitem[\protect\astroncite{Sacks et~al.}{1989}]{Sacks1989}
Sacks, J., Welch, W., Mitchell, T., and Wynn, H. (1989).
\newblock Design and analysis of computer experiments.
\newblock {\em Statistical Science}, 4(4):409--423.

\bibitem[\protect\astroncite{Schaback}{1995}]{Schaback1995}
Schaback, R. (1995).
\newblock Error estimates and condition numbers for radial basis function
  interpolation.
\newblock {\em Advances in Computational Mathematics}, 3(3):251--264.

\bibitem[\protect\astroncite{Schaback}{1999}]{Schaback1999}
Schaback, R. (1999).
\newblock Improved error bounds for scattered data interpolation by radial
  basis functions.
\newblock {\em Mathematics of Computation}, 68(225):201--216.

\bibitem[\protect\astroncite{Schaback}{2000}]{Schaback2000}
Schaback, R. (2000).
\newblock A unified theory of radial basis functions: Native {H}ilbert spaces
  for radial basis functions {II}.
\newblock {\em Journal of Computational and Applied Mathematics},
  121(1--2):165--177.

\bibitem[\protect\astroncite{Schaback}{2018}]{Schaback2018}
Schaback, R. (2018).
\newblock Superconvergence of kernel-based interpolation.
\newblock {\em Journal of Approximation Theory}, 235:1--19.

\bibitem[\protect\astroncite{Scheuerer}{2010}]{Scheuerer2010}
Scheuerer, M. (2010).
\newblock Regularity of the sample paths of a general second order random
  field.
\newblock {\em Stochastic Processes and their Applications},
  120(10):1879--1897.

\bibitem[\protect\astroncite{Scheuerer et~al.}{2013}]{Scheuerer2013}
Scheuerer, M., Schaback, R., and Schlather, M. (2013).
\newblock Interpolation of spatial data -- {A} stochastic or a deterministic
  problem?
\newblock {\em European Journal of Applied Mathematics}, 24(4):601--629.

\bibitem[\protect\astroncite{Shi and Wang}{2008}]{Shi2008}
Shi, J. and Wang, B. (2008).
\newblock Curve prediction and clustering with mixtures of {G}aussian process
  functional regression models.
\newblock {\em Statistics and Computing}, 18(3):267--283.

\bibitem[\protect\astroncite{Stein}{1993}]{Stein1993}
Stein, M.~L. (1993).
\newblock Spline smoothing with an estimated order parameter.
\newblock {\em The Annals of Statistics}, 21(3):1522--1544.

\bibitem[\protect\astroncite{Stein}{1999}]{Stein1999}
Stein, M.~L. (1999).
\newblock {\em Interpolation of Spatial Data: {S}ome Theory for Kriging}.
\newblock Springer.

\bibitem[\protect\astroncite{Steinwart}{2019}]{Steinwart2017}
Steinwart, I. (2019).
\newblock Convergence types and rates in generic {K}arhunen-{L}oève expansions
  with applications to sample path properties.
\newblock {\em Potential Analysis}, 51(3):361--395.

\bibitem[\protect\astroncite{Steinwart and Scovel}{2012}]{SteinwartScovel2012}
Steinwart, I. and Scovel, C. (2012).
\newblock {M}ercer's theorem on general domains: On the interaction between
  measures, kernels, and {RKHS}s.
\newblock {\em Constructive Approximation}, 35(3):363--417.

\bibitem[\protect\astroncite{Sun et~al.}{2018}]{Sun2018}
Sun, S., Zhang, G., Wang, C., Zeng, W., Li, J., and Grosse, R. (2018).
\newblock Differentiable compositional kernel learning for {G}aussian
  processes.
\newblock In {\em 35th International Conference on Machine Learning},
  volume~31, pages 4828--4837.

\bibitem[\protect\astroncite{Szab\'{o} et~al.}{2013}]{Szabo2013}
Szab\'{o}, B., van~der Vaart, A.~W., and van Zanten, J.~H. (2013).
\newblock Empirical {B}ayes scaling of {G}aussian priors in the white noise
  model.
\newblock {\em Electronic Journal of Statistics}, 7:991--1018.

\bibitem[\protect\astroncite{Szab\'{o} et~al.}{2015}]{Szabo2015}
Szab\'{o}, B., van~der Vaart, A.~W., and van Zanten, J.~H. (2015).
\newblock Frequentist coverage of adaptive nonparametric {B}ayesian credible
  sets.
\newblock {\em The Annals of Statistics}, 43(4):1391--1428.

\bibitem[\protect\astroncite{Teckentrup}{2019}]{Teckentrup2019}
Teckentrup, A.~L. (2019).
\newblock Convergence of {G}aussian process regression with estimated
  hyper-parameters and applications in {B}ayesian inverse problems.
\newblock {\em arXiv:1909.00232v2}.

\bibitem[\protect\astroncite{Triebel}{2006}]{Triebel2006}
Triebel, H. (2006).
\newblock {\em Theory of Function Spaces III}.
\newblock Birkhäuser Basel.

\bibitem[\protect\astroncite{Tuo et~al.}{2020}]{TuoWangWu2019}
Tuo, R., Wang, W., and Wu, C. F.~J. (2020).
\newblock On the improved rates of convergence for {M}atérn-type kernel ridge
  regression, with application to calibration of computer models.
\newblock {\em arXiv:2001.00152v1}.

\bibitem[\protect\astroncite{van~der Vaart and van
  Zanten}{2008}]{VaartZanten2008}
van~der Vaart, A.~W. and van Zanten, J.~H. (2008).
\newblock {\em Reproducing Kernel {H}ilbert spaces of {G}aussian Priors},
  volume~3 of {\em IMS Collections}, pages 200--222.
\newblock Institute of Mathematical Statistics.

\bibitem[\protect\astroncite{van~der Vaart and van
  Zanten}{2011}]{VaartZanten2011}
van~der Vaart, A.~W. and van Zanten, J.~H. (2011).
\newblock Information rates of nonparametric {G}aussian process methods.
\newblock {\em Journal of Machine Learning Research}, 12:2095--2119.

\bibitem[\protect\astroncite{Wang et~al.}{2006}]{Wang2006}
Wang, J., Hertzmann, A., and Fleet, D.~J. (2006).
\newblock Gaussian process dynamical models.
\newblock In {\em Advances in Neural Information Processing Systems 19}, pages
  1441--1448.

\bibitem[\protect\astroncite{Wang}{2020}]{Wang2020}
Wang, W. (2020).
\newblock On the inference of applying {G}aussian process modeling to a
  deterministic function.
\newblock {\em arXiv:2002.01381v1}.

\bibitem[\protect\astroncite{Wendland}{2005}]{Wendland2005}
Wendland, H. (2005).
\newblock {\em Scattered Data Approximation}.
\newblock Cambridge University Press.

\bibitem[\protect\astroncite{Wendland and Rieger}{2005}]{WendlandRieger2005}
Wendland, H. and Rieger, C. (2005).
\newblock Approximate interpolation with applications to selecting smoothing
  parameters.
\newblock {\em Numerische Mathematik}, 101(4):729--748.

\bibitem[\protect\astroncite{Wynne et~al.}{2020}]{Wynne2019}
Wynne, G., Briol, F.-X., and Girolami, M. (2020).
\newblock Convergence guarantees for {G}aussian process approximations under
  several observation models.
\newblock {\em arXiv:2001.10818v1}.

\bibitem[\protect\astroncite{Xu and Stein}{2017}]{XuStein2017}
Xu, W. and Stein, M.~L. (2017).
\newblock Maximum likelihood estimation for a smooth {G}aussian random field
  model.
\newblock {\em SIAM/ASA Journal on Uncertainty Quantification}, 5(1):138--175.

\bibitem[\protect\astroncite{Yang et~al.}{2013}]{Yang2013}
Yang, K., Keat~Gan, S., and Sukkarieh, S. (2013).
\newblock A {G}aussian process-based {RRT} planner for the exploration of an
  unknown and cluttered environment with a {UAV}.
\newblock {\em Advanced Robotics}, 27(6):431--443.

\end{thebibliography}

\end{document}